\documentclass{amsart}

\title[Closed Lie ideals and center
 of generalized group
algebras]{On closed Lie ideals and center of generalized group algebras}

\author[V. P. Gupta, R. Jain and B. Talwar]{Ved
  Prakash Gupta, Ranjana Jain and Bharat Talwar}

\address{Ved Prakash
  Gupta, School of Physical Sciences, Jawaharlal Nehru University, New
  Delhi} \email{vedgupta@mail.jnu.ac.in, ved.math@gmail.com}

\address{Ranjana Jain,
  Department of Mathematics, University of Delhi, Delhi}
\email{rjain@maths.du.ac.in}

\address{Bharat Talwar, Department of
  Mathematics, University of Delhi, Delhi}
\email{btalwar.math@gmail.com}

\thanks{The third named author was supported by a Junior Research Fellowship
	of CSIR with file number 09/045(1442)/2016-EMR-I}

\usepackage{relsize}
\usepackage{amsmath,amsfonts,amssymb,amsthm,mathrsfs}
\usepackage{hyperref}
\usepackage{cleveref}
\usepackage[margin=3.7cm]{geometry}

\usepackage{setspace}
\usepackage{xcolor}

\subjclass[2010]{22D15, 43A20, 46M05}

\allowdisplaybreaks[1]

\numberwithin{equation}{section}
\newtheorem{theorem}{\bf Theorem}[section]
\newtheorem{lemma}[theorem]{\bf Lemma}
\newtheorem{cor}[theorem]{\bf Corollary}
\newtheorem{remark}[theorem]{\bf Remark}
\newtheorem{prop}[theorem]{\bf Proposition}

\newtheorem{example}[theorem]{\bf Example}

\newcommand{\oop}{\widehat\otimes}

\newcommand{\omin}{\otimes^{\min}}

\newcommand{\obp}{\otimes^\gamma}
\newcommand{\oi}{\otimes^\lambda}

\newcommand{\ra}{\rightarrow}
\newcommand{\ot}{\otimes}

\newcommand{\Z}{\mathcal{Z}}

\newcommand{\C}{\mathbb{C}}
\newcommand{\N}{\mathbb{N}}

\newcommand{\R}{\mathbb{R}}

\newcommand{\ol}{\overline}
\newcommand{\mcal}{\mathcal}

\newcommand{\A}{L^1(G,A)}

\newcommand{\BA}{Banach algebra}

\newcommand{\norm}[1]{\| #1 \|}

\begin{document}

	\begin{abstract}
 For any locally compact group $G$ and any Banach algebra $A$, a
 characterization of the closed Lie ideals of the generalized group
 algebra $L^1(G,A)$ is obtained in terms of left and right actions by
 $G$ and $A$. In addition, when $A$ is unital and $G$ is an ${\bf
   [SIN]}$ group, we show that the center of $L^1(G,A)$ is precisely
 the collection of all center valued functions which are constant on
 the conjugacy classes of $G$. As an application, we establish that
 $\mathcal{Z}(L^1(G) \obp A)= \mathcal{Z}(L^1(G)) \obp \mcal{Z}(A)$,
 for a class of groups and Banach algebras. And, prior to these, for
 any finite group $G$, the Lie ideals of the group algebra $\C[G]$ are
 identified in terms of some canonical spaces determined by the
 irreducible characters of $G$.
        \end{abstract}
\keywords{Lie ideals, group algebras, generalized group algebras,
  projective tensor product, center}

\maketitle

\section{Introduction}

An associative algebra $A$ inherits a canonical Lie algebra structure
with respect to the Lie bracket given by the commutator $[x,y]:=xy - yx$. And, a
subspace $L$ of $A$ is said to be a {\it Lie ideal} if $ [L,A]
\subseteq L$, where
\[
  [L,A] := \text{span}\big\{[x,a]: x \in
  L, a \in A \big\}.
  \] 
The project of analysis of ideals of various tensor products
of operator algebras has attracted some well known operator
algebraists and a substantial amount of work has been done in this
direction - see, for instance, \cite{Guichardet, ass, arch2,  
  jk11, JK-edin, GJ2, TJ} and the references therein.

 Analysis of Lie ideals of associative rings was initiated as early as
 in 1955 by Herstein (\cite{herstein, herstein2}), for whom the
 motivation was purely algebraic, and was followed up enthusiastically
 by Herstein himself and many other algebraists. On the other hand,
 the study of closed Lie ideals of operator algebras is primarily
 motivated by the well understood relationship that exists between
 commutators, projections and closed Lie ideals in $C^*$-algebras -
 see \cite{Ped80, Marcoux2}.  Given its relevance, a good amount of
 work has also been done to examine the Lie ideals of pure as well as
 Banach algebras - see \cite{ Miers, FMS, Marcoux,  MM, paulsen, bks,
   Marcoux2, robert}.

 However, unlike the ideals of tensor products of operator algebras,
 not much was known about the closed Lie ideals of various tensor
 products of operator algebras. The pioneering works in this direction
 appear mainly in the works of Marcoux (\cite{Marcoux}) and Bre\v{s}ar
 et al.~(\cite{bks}). Marcoux, basically, identified all the closed
 Lie ideals of $A \omin C(X)$ for a UHF $C^*$-algebra $A$ and a
 compact space $X$. And, in 2008, relying heavily on the Lie ideal
 structure of tensor products of pure algebras, Bre\v{s}ar et
 al.~proved that for a unital Banach algebra (resp., $C^*$-algebra)
 $A$, the closed Lie ideals of the Banach space projective tensor
 product $A \obp K(H)$ and of the Banach space injective tensor
 product $A \oi K(H)$ (if it is a Banach algebra) (resp., of $A \omin
 K(H)$) are precisely the closed ideals.

 Motivated by these works of Marcoux and Bre\v{s}ar et al., we
 analyzed the structure of closed Lie ideals of certain ($C^*$ and
 other) tensor products of $C^*$-algebras in \cite{GJ, brv, TJ}, which
 also includes a generalization of Marcoux's result. This article is
 basically a continuation of {this project to the analysis of} closed
 Lie ideals of the projective tensor product of certain Banach
 algebras. More specifically, in this article, we mainly analyze the
 closed Lie ideals of the projective tensor product $L^1(G)\obp A$
 (which is also identified with the generalized group algebra $L^1(G,
 A)$), for any locally compact group $G$ and any Banach algebra $A$.
 The main result in this direction is the following
 characterization:\smallskip

  \noindent {\bf \Cref{LieidealsinL1G}.} {\em Let $G$ be a locally
    compact group and $A$ be a Banach algebra. Then, a closed
    subspace $L$ of $\A$ is a Lie ideal if and only if
	\[
	\Delta(x^{-1})(f\cdot x^{-1})a - a(x\cdot f ) \in L
	\]
	for every $f \in L$, $x \in G$ and $a \in A$, where $\Delta$
        denotes the modular function of group $G$.  }
\smallskip

 It would be appropriate to mention here that there is a subtle
 difference between the perspectives of some earlier works in this
 direction and of this article. In a nutshell, the focus of the
 articles \cite{bks, GJ, brv, Marcoux, TJ} was on characterizing and
 identifying closed Lie ideals of some Banach algebras and various
 tensor products of $C^*$-algebras by exploiting, among others, the
 notion of the {\em Lie normalizer} of an ideal, whereas the (above)
 characterization that we obtain in this article is essentially
 motivated by a technique that was considered by Laursen
 (\cite[Theorem 2.2]{laursen}) to characterize the closed (left and
 right) ideals of the generalized group algebra. Section 3
 is devoted to this characterization, few generalities and some
 applications.
\smallskip

 As the title suggests, the other important aspect of this article is
 the analysis of the center of a generalized group algebra and
 Section 4 is completely devoted to this theme. There already exist
 some satisfying results related to the center of a group algebra in
 literature - see \cite{mosak, mosak1, losert}. For instance, it is
 known that for a locally compact group $G$, the center
 $\mcal{Z}(L^1(G))$ is non-trivial if and only if $G$ is an ${\bf
   [IN]}$ group (\cite{mosak1}). Its analogue for generalized group
 algebras holds as is shown in \Cref{centre}. Also, for an ${\bf
   [IN]}$ group $G$, $f \in L^1(G)$ is central if and only if $f$ is
 constant on the conjugacy classes of $G$ (\cite{losert}). Taking cue
 from this, we prove the following:\smallskip

 \noindent {\bf \Cref{constantonconjugacy}} {\em Let $G$ be a locally
  compact ${\bf[IN]}$ group and $A$ be a unital Banach algebra. If
  either $G$ is an ${\bf [SIN]}$ group or $\Z(A) =\C1 $, then
\[
	\mcal{Z}(\A) = \{f\in L^1(G, \mcal{Z}(A)): f \ \text{is
	  constant on the conjugacy classes of } G \}.
\]}       
On the other hand, for any two $C^*$-algebras $A$ and $B$, Haydon and
Wassermann (\cite{HW}) showed that $\mathcal{Z}(A \omin B) =
\mathcal{Z}(A) \omin \mathcal{Z}(B)$, which was later extended to any
$C^*$-tensor product by Archbold (\cite{arch1}).  Similar isometric
identification was obtained for the Haagerup tensor product of unital
$C^*$-algebras by Smith (\cite{smith}) and then for non-unital
$C^*$-algebras by Allen, Sinclair and Smith (\cite{ass}). Then, the
second named author along with Kumar (\cite{jk08}) obtained a
continuous (not necessarily isometric) isomorphism between
$\mathcal{Z}(A \oop B)$ and $ \mathcal{Z}(A) \oop \mathcal{Z}(B)$,
where $\oop$ denotes the operator space projective tensor product;
and, very recently, the first and the second named authors established
an isometric isomorphism for the Banach space projective tensor
product in \cite{GJ2}. However, such an identification is out of reach
for Banach algebras in full generality. For instance, if we take the
pathological example of a  Banach algebra $A$ with the
trivial multiplication, i.e., $ab=0$ for all $a, b \in A$, then
$\Z(L^1(G) \obp A) = L^1(G) \obp A$ but $\Z(L^1(G)) \obp \Z(A)\ (= (0)
\obp A)$ is trivial whenever $G$ is not an ${\bf [IN]}$
group. Interestingly when $A$ is unital, as an application of above
mentioned results, we obtain the following:\smallskip

\noindent{\bf \Cref{centerdistribution}.}  {\em 
  Let $G$ be a locally compact group and $A$ be a unital Banach
 algebra. If either
 \begin{enumerate}
\item $\mathcal{Z}(L^1(G)) $ is complemented in $L^1(G)$ by a
  projection of norm one with either $G$ discrete or $\Z(A) = \C1_A$,
  or,
\item $G$ is an ${\bf [FIA]}^-$ group, then
 \end{enumerate}
\[
\mathcal{Z}(L^1(G)) \obp \mcal{Z}(A) = \mathcal{Z}(L^1(G) \obp
        A).
\]
}

Prior to obtaining the above results, we identify all the Lie ideals
of the group algebra of a finite group in terms of some canonical
spaces determined by the irreducible characters of the group, in
Section 2.

\section{Lie ideals of finite group algebras}

To begin with, in this section, we identify the Lie ideals of the
(complex) group algebra of a finite group.  Since the group algebra of
any finite group is a finite direct sum of matrix algebras, we first
establish a basic result concerning Lie ideals of such algebras.

	 It is tempting to assume that for a direct sum of Banach
         algebras $A=\oplus_i A_i$, a closed subspace $L$ of $A$ is a
         Lie ideal if and only if $L = \oplus_i L_i$ for some closed
         Lie ideals $L_i$ in $A_i$.  However, this is not true as we
         shall see below.

For $n \in \mathbb{N}$, let $M_n$ denote the set of all $n \times n$
complex matrices; $I_n$ and $0_n$ denote its identity and zero
matrices, respectively; $O_{n}:=\{0_{n}\}$  and $sl_n$ denote the set
of all trace zero matrices in $M_n$. It is known that $M_n$ possesses
precisely $4$ Lie ideals, namely, $(0)$, $\C I_n$, $sl_n$ and $M_n$
(see \cite[Theorem 5]{herstein}). However, there are uncountably many
Lie ideals in the direct sums of matrix algebras, as we see from the
following observation.

\begin{prop} \label{Lieidealsofdirectsum}
	A subspace $L$ of $A:= \oplus_{i=1}^k M_{n_i}$ is a Lie ideal
        if and only if it satisfies
        \[
        \oplus_{i=1}^k \delta_i\, sl_{n_i} \subseteq
        L \subseteq \oplus_{i=1}^k \tilde\delta_i M_{n_i} + \oplus_{i=1}^k \C
        I_{n_i},
        \]
        for $\delta_i =\tilde\delta_i\in \{0, 1\}$, $1 \leq i \leq
        k$. 
\end{prop}
\begin{proof}
 It is known that a closed subspace $L$ of $C^*$-algebra $A$ is a Lie
 ideal if and only if there exists a closed ideal $J$ in $A$ such that
 $[J,A] \subseteq L \subseteq N(J)$, where $N(J):=\{x\in A: [x,
   a]\in J\text{ for all } a \in A\}$ (see \cite[Corollary 5.26, Theorem
   5.27]{bks}). Also, $J = \oplus_{i=1}^k J_{n_i}$, where $J_{n_i} =
 O_{n_i}$ or $M_{n_i}$ for $1 \leq i \leq k$, so that $[J,A] =
 \oplus_{i=1}^k[J_{n_i} ,M_{n_i}]$.  Then, note that $[J_{n_i}
   ,M_{n_i}] = sl_{n_i}$ when $J_{n_i} = M_{n_i}$ and $[J_{n_i}
   ,M_{n_i}] = O_{n_i}$ otherwise \cite[Theorem 5]{herstein}. The
 result now follows from the fact that $N(J) = J + \oplus_{i=1}^k \C
 I_{n_i}$, which can be verified easily.
\end{proof}

Recall that for a finite group $G$, the vector space
$\C[G]:=\{\sum_{x\in G} a_x x: a_x \in \C\}$ is a unital $*$-algebra with
multiplication induced by group multiplication and involution obtained
as conjugate linear extension of the map $G \ni x \mapsto x^*:=
x^{-1}\in G$, and is known as the {\em group algebra} of $G$.

A unitary representation $\pi: G \ra \mathcal{U}(V)$  of $G$ (on an inner
product space $V$ - also called a $G$-module and - is written simply
as $(V, \pi)$) can be linearly extended to  a unital $*$-representation
$\pi:\C[G]\ra B(V)$ of $\C[G]$, and, conversely, a unital
$*$-representation of $\C[G]$ restricts to a unitary representation of
$G$ in a bijective way. Note that an arbitrary finite dimensional
(complex) representation $\rho: G \ra GL(V)$ can be treated as a
unitary representation with respect to the inner product $\langle v, w
\rangle:=\sum_x\langle \rho_x(v), \rho_x(w)\rangle_0$, where $\langle
\cdot, \cdot \rangle_0$ is any inner product on $V$.

The left regular representation $\lambda: G \ra
\mathcal{U}\big(\ell^2(G)\big)$ of $G$ given by \( \lambda_x(\xi)(y)=
\xi(x^{-1}y)\ \text{for } x, y \in G, \xi \in \ell^2(G), \) extends to
an injective $*$-representation $\lambda: \C[G] \ra
B\big(\ell^2(G)\big)$. As a consequence, $\C[G]$ inherits a $C^*$-norm given by
$\|x\| = \|\lambda_x\|$ for $x \in \C[G]$. Being finite dimensional,
$\C[G]$ is a $C^*$-algebra with this norm and this is the unique
$C^*$-norm on $\C[G]$. If $G$ has $k$ conjugacy classes, say, $\{C_j :
1 \leq j \leq k\}$, then the center of $\C[G]$ is a $k$-dimensional
$*$-subalgebra with a basis given by $\{z_j:=\sum_{x\in C_j} x : 1
\leq j \leq k\}$.  Thus, from the structure theorem of (unital) finite
dimensional $C^*$-algebras, there exists  an algebra
  $*$-isomorphism $\psi: \C[G] \ra \oplus_{j=1}^k
M_{n_j}$, where each $n_j$ is the dimension of an irreducible unital
$*$-representation of $\C[G]$ (and hence of $G$).

\begin{remark} 
An algebra isomorphism from  $\C[G]$ onto $ \oplus_{j=1}^k M_{n_j}$ can also
be obtained from purely algebraic methods. For instance,  by
Artin-Molien-Wedderburn  Theorem, there exists an algebra isomorphism $\psi :
\C[G] \ra \oplus_{ j=1}^k M_{n_j}$, where each $n_j$ is the dimension of
an irreducible $G$-module.
\end{remark}

For any finite dimensional complex representation $(V, \pi)$ of $G$,
its character  $\chi : G \ra \C$ is given by $\chi(x):=
\text{trace}(\pi(x))$, where trace is not normalized and is defined
with respect to one (equivalently, any) basis of $V$.  Each such
character $\chi$ extends linearly to a map $\widetilde{\chi}: \C[G]
\ra \C$ and, by linearity of the trace map, we have $\widetilde{\chi}(z) =
\text{trace}(\pi(z))$ for all $z \in \C[G]$. Two representations of
$G$ are equivalent if and only if they have same characters (see
\cite[Theorem 14.21]{jamesliebeck}).  The character of an irreducible
representation (equivalently, all equivalent irreducible
representations) of $G$ is, in short, also called an {\em irreducible
  character} of $G$.  If $G$ has $k$ conjugacy classes, then $G$ has
precisely $k$ distinct irreducible characters. We assert the following
identification of Lie ideals of group algebras.

\begin{theorem}\label{LieidealsinCG}
  Let $G$ be a finite group with $k$ conjugacy classes. Let $\{\chi_j:
  1 \leq j \leq k\}$ denote the set of irreducible characters of $G$
  and $n_j$ denote the dimension of any representation corresponding
  to $\chi_j$. Let \( \omega_j :=\frac{n_j}{|G|}\sum_{x \in G}
  \chi_j(x^{-1}) x \in \C[G] \) and $K_j$ denote the ideal generated
  by $\omega_j$ for $1 \leq j \leq k$. Then,
  \begin{enumerate}
\item  $\{\omega_j: 1 \leq j \leq k\}$
 is the set of minimal central projections of $\C[G]$, and
  \item $ \{K_j:1\leq j \leq k\}$ is the set of distinct non-zero
    minimal ideals of $\C[G]$.
\end{enumerate}
In particular, a subspace
 $L$ of $\C[G]$ is a Lie ideal if and only if
        \[
        \sum_{j=1}^k \delta_j \ker(\widetilde{\chi_j})\omega_j
        \subseteq L \subseteq \sum_{j=1}^k \tilde\delta_j K_j +
        \mathrm{span}\{\omega_j : 1 \leq j \leq k \},
        \]
        for $\delta_j = \tilde\delta_j\in \{0,1\}$, $1 \leq j \leq
        k$. %({\bf B: I do not understand the role of $\tilde\delta_j$
          %V: Since $\delta_j$ is like a variable, it can't take same
          %values on both sides of the equation...}.)
  \end{theorem}

In order to prove \Cref{LieidealsinCG},
in view of Proposition \ref{Lieidealsofdirectsum} and the isomorphism
$\psi: \C[G] \ra \oplus_j M_{n_j}$, it suffices to make the following
observations, which are most probably folklore. We include the details
 just for the sake of completeness.

\begin{prop} \label{idealsofCG}
  Let $G$, ${\chi_j}$, $\omega_j$ and $\widetilde{\chi_j}$ for $1 \leq
  j \leq k$ be as in \Cref{LieidealsinCG} and $\psi: \C[G] \ra
  \oplus_{j=1}^k M_{n_j}$ be a $*$-isomorphism. Then,
\begin{enumerate}
\item $\omega_j = \psi^{-1}((0_{n_1}, \ldots, 0_{n_{j-1}}, I_{n_j},
  0_{n_{j+1}}, \ldots, 0_{n_k}))$; so that
  $\psi^{-1}\big(\widetilde{M}_{n_j}\big) = K_j$, and
    \item $\ker (\widetilde{\chi_j}) = \psi^{-1}\big(M_{n_1} \oplus \cdots
  \oplus M_{n_{j-1}}\oplus sl_{n_j}\oplus M_{n_{j+1}} \oplus \cdots
  \oplus M_{n_k}\big)$; so that $\ker(\widetilde{\chi_j})\omega_j =
  \psi^{-1}\big(O_{n_1}\oplus\cdots\oplus O_{n_{j-1}}\oplus
  sl_{n_j}\oplus O_{n_{j+1}}\oplus \cdots \oplus O_{n_k}\big)$
\end{enumerate}
for all $1 \leq j \leq k$.  \end{prop}

\begin{proof}

Note that the $C^*$-algebra $\oplus_j M_{n_j}$ possesses a natural
inner product given by
\[
\langle(X_1, \ldots, X_k), (Y_1, \ldots,
Y_k)\rangle =\sum_{j=1}^k \text{trace}_{M_{n_j}}\big(Y_j^* X_j),
\]
where trace is not normalized. Let $\widetilde{M}_{n_j}:=
O_{n_1}\oplus \cdots \oplus O_{n_{j-1}}\oplus M_{n_j}\oplus
O_{n_{j+1}}\oplus \cdots\oplus O_{n_k}$. Then,
$\widetilde{M}_{n_j}^{\perp}= M_{n_1}\oplus \cdots \oplus
M_{n_{j-1}}\oplus O_{n_j}\oplus M_{n_{j+1}}\oplus \cdots\oplus
M_{n_k}$.  Also, let $U(n_j)$ denote the set of all unitaries in
$M_{n_j}$.

  The proof now is mainly a book keeping exercise.  Let $A:=\oplus_{i=1}^k
M_{n_i}$.  For each $1 \leq j \leq k$, consider the $*$-representation
$\theta_j: A \ra M_{n_j} = B(\C^{n_j})$ given by $\theta_j((X_1,
\ldots, X_k)) = X_j$. Then, $\{\theta_j: 1\leq j\leq k\}$ constitutes
a complete set of inequivalent irreducible $*$-representations of the
$C^*$-algebra $A$.

Taking $\pi_j= \big(\theta_j \circ \psi\big)_{|_{G}} :G \ra U(n_j)$,
we observe that $\{\pi_j : 1 \leq j \leq k\}$ forms a complete set of
inequivalent irreducible representations of $G$; so that \( \chi_j =
\chi_{\pi_j}\) for all $1 \leq j \leq k$.

Next, consider the canonical $*$-representation $\Gamma: A \ra B(A)$
of $A$ given by $\Gamma(a)(b)= ab$ for $ a, b \in A$, with respect to
the canonical inner product on $A$ described in the preceding
paragraph.  Then, $\widetilde{M}_{n_j}$ is an $A$-submodule of $A$. To
make it explicit, let this $A$-module be denoted by $(
\widetilde{M}_{n_j}, \rho_j)$. Then, $\rho_j \cong \big(\theta_j
\oplus \cdots \oplus \theta_j \big)$ ($n_j$-fold direct sum) as
$A$-modules for all $1\leq j \leq k$.\smallskip

$(1)$: Let $e_j := \psi^{-1}((0_{n_1}, ..., 0_{n_{j-1}}, I_{n_j},
0_{n_{j+1}}, ... ,0_{n_k}))$ for $ 1 \leq j \leq k$. Consider $\C[G]$
with the inner product induced by the isomorphism $\psi$. Then, again
via $\psi$ and $\Gamma$, $\C[G]$ becomes a $\C[G]$-module
(equivalently, $G$-module) and the action is again the left
multiplication on $\C[G]$. As a result,
$\psi^{-1}(\widetilde{M}_{n_j})$ and $
\psi^{-1}(\widetilde{M}_{n_j}^{\perp})$ are $\C[G]$-submodules of
$\C[G]$.

Now, let $\varphi_j:=
(\psi^{-1})_{|_{\widetilde{M}_{n_j}}}:\widetilde{M}_{n_j} \ra
\psi^{-1}\big( \widetilde{M}_{n_j} \big) $ for $1 \leq j \leq
k$. Then, $\varphi_j$ is a $G$-module isomorphism between
$\widetilde{M}_{n_j}$ and $ \psi^{-1}\big( \widetilde{M}_{n_j} \big)$;
thus, the $G$-module $\psi^{-1}\big( \widetilde{M}_{n_j} \big) $ is
also isomorphic to $n_j$-copies of $\pi_j$ (:= $(\theta_j \circ
\psi)_{|_G}$).  So, its character is given by $n_j \chi_j$ (see, for
instance, \cite[Page 141]{jamesliebeck}) and
$\psi^{-1}(\widetilde{M}_{n_j})$ and $
\psi^{-1}(\widetilde{M}_{n_j}^{\perp})$ contain no isomorphic
irreducible $\C[G]$-submodules. Also, $\C[G] =
\psi^{-1}(\widetilde{M}_{n_j}) \oplus
\psi^{-1}(\widetilde{M}_{n_j}^{\perp})$ and $1 = e_j + \tilde{e}_j$,
where $\tilde{e_j}:= \psi^{-1}((I_{n_1}, ..., I_{n_{j-1}}, 0_{n_j},
I_{n_{j+1}},$ $ ... , I_{n_k}))\in \widetilde{M}_{n_j}^{\perp}$.  So,
by \cite[Proposition 14.10]{jamesliebeck}, we obtain $e_j =
\mathsmaller{\frac{1}{|G|}}\mathsmaller{\sum}_{x \in G}
n_j\chi_j(x^{-1}) x = \omega_j$.\smallskip

  $(2)$: Let $(X_1, \ldots, X_k) \in M_{n_1}\oplus \cdots \oplus
M_{n_{j-1}}\oplus sl_{n_j}\oplus M_{n_{j+1}}\oplus\cdots \oplus
M_{n_k}$. Then,
\begin{eqnarray*}
  \widetilde{\chi}_j\big(\psi^{-1} ((X_1, \ldots, X_k) )\big) & = &
  \text{trace} \big( \theta_j \circ \psi\circ \psi^{-1} ((X_1, \ldots,
  X_k) )\big)\\ & = & \text{trace}(X_j)\\ & = & 0.
\end{eqnarray*}
So, $\psi^{-1}\big(M_{n_1}\oplus \cdots \oplus M_{n_{j-1}}\oplus
sl_{n_j}\oplus M_{n_{j+1}}\oplus\cdots \oplus M_{n_k}\big) \subseteq
\ker(\widetilde{\chi_j})$. And, both subspaces, being of co-dimension
1 in $\C[G]$, are therefore equal. \end{proof}

As an application, we provide a non-central Lie ideal in $\C[D_6]$
which is not an ideal.
\begin{example}\label{Lieidealnotideal}
	Let $D_6=\langle r, s : r^3 = 1, s^2 = 1, srs = r^{-1}\rangle
        $ denote the Dihedral group consisting of symmetries of an
        equilateral triangle. $D_6$ has three conjugacy classes and
        its irreducible characters are well understood - see \cite[Page
          121]{jamesliebeck}. If $\{ \chi_1, \chi_2, \chi_3 \}$ forms
        a complete set of characters of irreducible representations of
        $D_6$, then $\chi_3$ can be taken to be the one that satisfies
        $\chi_3(e) =2$, $\chi_3(r) = \chi_3(r^2) = -1$ and $\chi_3(s)
        = \chi_3(rs) = \chi_3(r^2s) = 0$.
        
        Let $L:= \ker(\widetilde{\chi_3})\omega_3$. Then, by
        \Cref{LieidealsinCG}, $L$ is a Lie ideal in $\C[D_6]$; and it
        is non-central. Also, it is routine to check that $L = \{
        c_1(r-r^2)+ c_2(s - r^2s) + c_3(rs -r^2s) : c_i \in \C \}$,
        so, being a 3-dimensional subspace, it can not be an ideal of
        $\C [D_6] \cong \C \oplus \C \oplus M_2$.
        \end{example}

\section{Closed Lie ideals of generalized group algebras}

First, recall that for a general measure space $(\Omega, \mathcal{M},
\mu)$ and a Banach space $X$, a function $f: \Omega \ra X$ is said to
be
\begin{enumerate}
\item  {\em simple $(\mathcal{M}, \mu)$-measurable}  if $f$ takes
  only finitely many values and $f^{-1}(x) \in \mathcal{M}$ for every $x
  \in X$.
\item {\em $(\mathcal{M}, \mu)$-measurable} if there exists a sequence $\{s_n\}$
of $X$-valued simple $(\mathcal{M}, \mu)$-measurable functions on $\Omega$ such
that $s_n\ra f$ a.e.  $[\mu]$.
\item {\em Bochner integrable} if there exists a sequence $\{s_n\}$ of
  $X$-valued simple $(\mathcal{M}, \mu)$-measurable functions on
  $\Omega$  with $\mu(\mathrm{supp}(s_n)) < \infty$ for all $n$ such that
  $s_n\ra f$ a.e.  $[\mu]$ and $\int_{\Omega}\|f - s_n\| d\mu \ra 0$,
  where $\|g\|$ denotes the scalar valued function $x \mapsto
  \|g(x)\|$.
\end{enumerate}
 It is known that $f$ is Bochner integrable if and only if $f$ is
 $(\mathcal{M}, \mu)$-measurable and $\|f\|\in L^1(\Omega, \mu)$.  For such
an $f$, its {\em Bochner integral} is defined as $\int_{\Omega} f(x) d\mu =
 \lim_n \int_{\Omega} s_n(x) d\mu$. And, the space of $X$-valued
 Bochner integrable functions on $\Omega$ is denoted by $L^1(\Omega,
 X)$. See \cite{ryan} for further details.\smallskip

We now come to the objects of present interest, namely, {\em generalized group
algebras}. As is standard, by a {\em locally compact group} we shall mean a
topological group which is both Hausdorff and locally compact.

Let $G$ be a locally compact group with a left Haar measure $m$. Then,
there exists a continuous group homomorphism $\Delta: G \rightarrow
\R_{>0}$ such that $m(Sx) =\Delta(x)m(S)$ for all $S \in B_G$ and $x
\in G$, where $B_G$ denotes the Borel $\sigma$-algebra of $G$.  This
group homomorphism is known as the {\em modular function} of $G$. It
is a well known fact that when $A$ is a Banach algebra, the space \(
L^1(G, A)\), known as a {\em generalized group algebra}, is a Banach
algebra with respect to the multiplication given by the convolution
$(f*g)(x) := \int_G f(xs)g(s^{-1})ds$ and norm given by $\|f\|_1=
\int_G \|f\|(x) dx$ - see \cite{ kaniuthbook, ryan} for details.

For every $a \in A$ and $f \in L^1(G)$, define $fa: G \ra A$ by
$(fa)(x) = f(x)a$ for all $x\in G$. It is easily seen that $fa \in \A$
for all $a \in A$ and $f \in L^1(G)$. We have the following well known
identification.

\begin{theorem}  \label{L1GA}   	
The map $L^1(G) \ot A \ni \sum_i f_i \ot a_i \mapsto \sum_i f_ia_i \in
\A$ extends to an isometric isomorphism from $L^1(G) \obp A$ onto
$L^1(G, A)$.

In particular,
 the subspace  $\mathrm{span}\{ fa : f \in L^1(G), a \in A  \}$ is dense in $\A$.
\end{theorem}

As per the requirements of this article, we discuss some useful
properties, whose analogues are well known for group algebras
(see, for instance, \cite{ kaniuthbook, palmer}).

  For each $x \in G$ and $f \in L^1(G,A)$, one defines $x \cdot f, f
  \cdot x: G \ra A$ by
  \[
  (x\cdot f)(y) :=
  f(x^{-1}y) \text{ and } (f\cdot x)(y) := f(yx)\ \text{for all } y \in G.
  \]
  It can be easily verified that $x\cdot f, f\cdot x \in \A$ for all
  $x \in G$ and $f \in \A$.

  \begin{lemma}\label{f-delta-relation}
   For any $f \in \A$, the mapping $G \ni x \mapsto \Delta
   (x^{-1})f(x^{-1}) \in A$ is Bochner integrable and
    \begin{equation}\label{f-inverse}
  \int_{G} f (x) dx = \int_G \Delta (x^{-1})f(x^{-1}) dx.
        \end{equation}
    Also, for any $y \in G$ and $f \in \A$, we have
    \begin{equation}\label{f-Delta}
    \int_G (f\cdot y)(x) dx = \Delta({y^{-1}}) \int_G f (x) dx.
\end{equation}
\end{lemma}
\begin{proof}
  Fix a sequence $\{s_n\} \subset \A$ of simple $(B_G, m)$-measurable functions
  converging almost everywhere to $f$ and satisfying $\lim\limits_{n\ra
    \infty} \int_G \|(f - s_n)(x)\| dx = 0$. Since the sequence of scalar functions
  $\{\|f  - s_n  \|\}$ is contained in $ L^1(G)$, from \cite[page
    1484]{palmer}, we see that the map $G\ni x \mapsto
  \|\Delta(x^{-1})\big( f(x^{-1}) - s_n(x^{-1})\big)\|\in \C$ is in
  $L^1(G)$, for all $n$, and
  \[
  \lim_n \int_G \| \Delta (x^{-1}) f (x^{-1}) - \Delta (x^{-1}) s_n
  (x^{-1})\| dx =\lim_n \int_G \|(f - s_n)(x)\| dx = 0.
  \]
 Thus, $$\int_G f(x) dx = \lim_n \int_G s_n(x) dx = \lim_n
\int_G  \Delta(x^{-1}) s_n(x^{-1}) dx = \int_G  \Delta(x^{-1}) f(x^{-1}) dx,$$
where the second equality holds because each $s_n$ is simple.

In order to prove (\ref{f-Delta}), fix a sequence $\{t_n\} \subset \A$
of simple $(B_G, m)$-measurable functions such that it converges
almost everywhere to $f\cdot y$ and $\lim\limits_{n\ra \infty} \int_G
\| (f\cdot y)(x) - t_n(x) \| dx = 0$. From \cite[page 1484]{palmer}
again, we obtain
\[
\int_G
(\| f \cdot y - t_n \| \cdot y^{-1})(x) dx = \Delta(y) \int_G \| f \cdot
y - t_n \|(x) dx
\]
for all $n$. In particular,
\begin{equation}\label{use}
\int_G \|f - t_n \cdot y^{-1}\|(x) dx =  \int_G (\|  f \cdot y - t_n \| \cdot y^{-1})(x) dx 
= \Delta(y) \int_G \|  f \cdot y - t_n \| (x) dx \rightarrow 0.
\end{equation}
Thus,
\begin{eqnarray*}
\int_G (f\cdot y)(x) dx & = & \lim_n \int_G t_n(x) dx \\ & = & \lim_n \int_G
t_n(xy^{-1}y) dx \\ & = & \Delta(y^{-1}) \lim_n \int_G t_n(xy^{-1})
dx\quad (\text{because each } t_n \text{ is simple)}\\
 & = & \Delta(y^{-1}) \lim_n \int_G (t_n \cdot y^{-1})(x) dx \\ 
 & = & \Delta(y^{-1}) \int_G f(x)dx. \qquad \qquad \text{(by \Cref{use})}
\end{eqnarray*}
This completes the proof.
\end{proof}

\begin{lemma} \label{cont}
 For each $f \in \A$, the maps $G\ni x \mapsto  x\cdot f,\, f
  \cdot x \in \A$ are continuous.
\end{lemma}
  
\begin{proof}
 Let  $\epsilon > 0$. Fix an $f_0 = \sum_{i=1}^{r} g_i
 a_i \in \text{span}\{g a : g \in L^1(G), a \in A \}$ such that $\|f -
 f_0\|_1 < \epsilon/3$.  For continuity of left multiplication, note
 that for any $x \in G$, we have
 \[
 \norm{x \cdot f - f}_1  \leq \norm{x \cdot f -x \cdot f_0 }_1 +
 \norm{ x \cdot f_0 - f_0 }_1 + \norm{f_0 -f}_1.
 \]
First, observe that
 \[
 \norm{ x \cdot   f_0 - f_0 }_1 
 = \mathsmaller{\|} x \cdot \big(\mathsmaller{\sum}_{i=1}^{r} g_i a_i\big) -
   \mathsmaller{\sum}_{i=1}^{r} g_i a_i \mathsmaller{\|}_1 
   \leq \mathsmaller{\sum}_{i=1}^{r} \norm{ (x\cdot g_i -g_i)}_1 \| a_i \|.
 \]
 Using continuity of left multiplication by elements of $G$ in
 $L^1(G)$ (see \cite[$\S A.4$]{kaniuthbook}), there exists a
 neighbourhood $V$ of $e$ such that $\sup_{i=1}^r\norm{ x \cdot g_i -
   g_i }_1 \leq \sup\{\|a_i\|: 1 \leq i \leq r\}^{-1}\epsilon/3r$ for
 all $x \in V$.  Then, using left invariance of Haar measure, we
 obtain $\norm{x \cdot f -x \cdot f_0 }_1 = \norm{f - f_0}_1 \leq
 \epsilon/3$ for all $x \in G$. Thus, we have $\|x \cdot f - f \|_1 <
 \epsilon$ for all $x \in V$. This proves continuity at $e$.

 Now, let $x \in G$ and $\{x_{\alpha}\}$ be a net converging to $x$ in
 $G$. Then, $x^{-1} x_{\alpha} \to e$ and, hence,
 \[
 \| x_{\alpha} \cdot f - x \cdot f \|_1 = \| x \cdot (x^{-1}
 x_{\alpha} \cdot f - f) \|_1 = \| (x^{-1} x_{\alpha} \cdot f - f)
 \|_1 \ra 0,
 \]
where the second equality follows by the left invariance of Haar measure.

 And, for continuity of the right multiplication  at $e$, following above
 steps, it suffices to show that $\norm{ f \cdot x - f_0 \cdot x}_1
 \leq \epsilon/3$ in a neighbourhood of $e$, which follows readily
 from the continuity of $\Delta$ and the relation
 \[
 \norm{ f
   \cdot x - f_0 \cdot x }_1 = \int_G \| \big((f-f_0) \cdot x \big) (y)\|dy =
 \Delta(x^{-1}) \int_G \| (f-f_0)  (y)\|dy,
 \]
where the last equality follows from the preceding lemma. Once
continuity at $e$ is established, then continuity at an arbibtrary $x
\in G$ is obtained on similar lines as above.
\end{proof}

That was all that we will require from the basic theory. We now proceed to
characterize the closed Lie ideals of generalized group algebras.

For every $a \in A$ and $f \in \A$, consider $af, fa: G \ra A$ given
by $(af)(x) = a f(x)$ and $(fa)(x) = f(x)a$ for all $x \in G$.  It is
easily seen that $af, fa \in \A$ for all $a \in A$ and $f \in \A$. In
1969, Laursen \cite{laursen} had shown if $A$ is a Banach algebra with
approximate identity, then a closed subspace $I$ of $\A$ is an ideal
if and only if $I$ is $A$-translation invariant, i.e.,
\[
f\cdot x,\  x \cdot f,\ af,\ fa \in I \ \text{for all } f\in I, a \in A, x \in G.
\]

It is easily seen that, for a countable discrete group $G$ with the
(unimodular) counting Haar measure and any Banach algebra $A$, a
closed subspace $L$ of $\A$ is a Lie ideal if and only if
\begin{equation}\label{discrete}
(f\cdot x^{-1})a - a (x \cdot f) \in L \ \text{for all } f \in L, a
  \in A, x \in G.
\end{equation}
Indeed, setting $\chi_x:=\chi_{\{x\}}$ for  $x \in G$, since
$\text{span} \{ a \chi_x : a \in A, x \in G\}$ is dense in $\A$, a
closed subspace $L$ of $\A$ is a Lie ideal if and only if
\[
f* (a\chi_x) - (a \chi_x)* f \in L\ \text{for all } f \in L, a \in A, x \in G.
\]
Then, note that
\[
 (f *(a \chi_x) \big) (t) = \mathsmaller{\sum}_{s\in G} f(ts) a \chi_x(s^{-1}) =
f(t x^{-1})a = \big( (f\cdot x^{-1})a\big) (t), \text{ i.e., } f *(a
\chi_x) = (f\cdot x^{-1})a;
\]
and, likewise, we have \( (a \chi_x) * f = a (x \cdot f)\). \smallskip

Unlike the collection $\{\chi_x: x \in G\}$ for a discrete group $G$,
there is no canonical way of identifying a copy of $G$ in $L^1(G)$ when
$G$ is not discrete. Thus, for an arbitrary locally compact group $G$,
we need a more analytical approach. We obtain the following general
form of the characterization observed in \ref{discrete}.

\begin{theorem}\label{LieidealsinL1G}
	 Let $G$ be a locally compact group
         and $A$ be a Banach algebra. Then, a closed subspace $L$ of
         $\A$ is a Lie ideal if and only if
	\begin{equation}\label{theorem-eqn}
	\Delta(x^{-1})(f\cdot x^{-1})a - a(x\cdot f ) \in L
	\end{equation}
	for every $f \in L$, $x \in G$ and $a \in A$.
\end{theorem}
\begin{proof} 	Let $L$ be a closed Lie ideal in $\A$, $f \in L$, $x \in G$
	and $a \in A$.  Since $L$ is closed, it is enough to show that
        $\Delta(x^{-1})(f\cdot x^{-1})a - a(x\cdot f ) \in
        \ol{L}$. Let $\epsilon > 0$.

For any Borel set $U$ in $G$ and $a \in A$, let $a_U:=a\chi_U$. If $U$
is compact, then $a_U \in \A$ for all $a \in A$.  So, it is enough to
find a compact  neighbourhood $V$ of $e \in G$ such that
\[
\| \Delta(x^{-1})(f\cdot x^{-1})a - a(x\cdot f ) -
\mathsmaller{\frac{1}{m(V)}} ( f \ast (a_{xV}) - (a_{xV}) \ast f )
\|_1 <\epsilon.
\]
Note that for an arbitrary compact symmetric neighbourhood $V$ of $e$,
$a \in A$ and $x \in G$, we have
\begin{eqnarray*}
  \lefteqn{ m(Vx^{-1}) \left \|\Delta(x^{-1})(f\cdot x^{-1})a - a(x\cdot f )
          - \frac{1}{m(V)} ( f \ast (a_{xV}) - (a_{xV}) \ast f ) \right \|_1}
  \\&         \leq &  \left \| m(Vx^{-1}) \Delta(x^{-1})(f\cdot x^{-1})a -
  \mathsmaller{\frac{m(Vx^{-1})}{m(V)}} f \ast (a_{xV}) \right \|_1 \\
  & & \qquad \qquad \qquad + \
          \left \| m(Vx^{-1}) a(x\cdot f ) -
          \mathsmaller{\frac{m(Vx^{-1})}{m(V)}} (a_{xV}) \ast f \right \|_1.\qquad \qquad \qquad (*)
        \end{eqnarray*}
Let $E:= \left\{ y \in G: \left\| \big( m(Vx^{-1}) (f\cdot x^{-1})a - f
\ast (a_{xV}) \big) (y) \right\| > 0\right\}$. It is easily seen that $E$
is a $\sigma$-finite Borel subset of $G$.   We have
        \begin{align*}
          & \left\| m(Vx^{-1}) \Delta(x^{-1})(f\cdot x^{-1})a -
          \frac{m(Vx^{-1})}{m(V)} f \ast (a_{xV}) \right\|_1\\ & =
          \Delta(x^{-1}) \int_G \left\| \big( m(Vx^{-1}) (f\cdot
          x^{-1})a - f \ast (a_{xV}) \big) (y) \right\| dy \qquad
          (\text{since } m(Vx^{-1}) = \Delta(x^{-1}) m(V)) \\ & =
          \Delta(x^{-1}) \int_E \left \| (f\cdot x^{-1})(y)a
          \int_{Vx^{-1}} ds - \int_G f(ys) a_{x V}(s^{-1}) ds \right
          \| dy\\ & = \Delta(x^{-1}) \int_E \left \| \int_{Vx^{-1}}
          (f\cdot x^{-1})(y)a ds - \int_{Vx^{-1}} f(ys) a\, ds \right
          \| dy \qquad (\text{since $s^{-1} \in xV \iff s \in
            Vx^{-1}$)}\\ & = \Delta(x^{-1})\int_E \left \|
          \int_{Vx^{-1}} ( (f \cdot x^{-1}) a - (f\cdot s) a )(y) ds
          \right \|dy\\ & \leq \Delta(x^{-1}) \int_E \int_{Vx^{-1}}
          \left \| ( (f \cdot x^{-1}) a - (f\cdot s) a )(y) \right \|
          ds dy\\ & = \Delta(x^{-1}) \int_{Vx^{-1}} \int_E \left \| (
          (f \cdot x^{-1}) a - (f\cdot s) a )(y) \right \| dy ds
          \qquad (\text{by Tonelli's Theorem})\\ & \leq \Delta(x^{-1})
          \int_{Vx^{-1}} \left \| (f \cdot x^{-1}) a - (f\cdot s) a
          \right \|_1 ds\\ & \leq \Delta(x^{-1}) m(Vx^{-1}) \sup_{s
            \in Vx^{-1}} \| (f \cdot x^{-1} - f\cdot s) a \|_1\\ & =
          \Delta(x^{-1}) m(Vx^{-1})\ \sup_{t \in V} \| (f \cdot x^{-1}
          - f \cdot (tx^{-1})) a \|_1 \\ & = \Delta(x^{-1}) m(Vx^{-1})
          \ \sup_{t \in V} \| (f \cdot x^{-1} - f \cdot
          (t^{-1}x^{-1})) a \|_1\ \qquad(\text{since $V$ is
            symmetric)}\\ & \leq \Delta(x^{-1}) m(Vx^{-1}) \ \sup_{t
            \in V} \| f \cdot x^{-1} - (f \cdot x^{-1})\cdot
          t^{-1}\|_1\| a \| ; \end{align*} and, on the other hand,
        considering the $\sigma$-finite Borel set $E':=\{ y \in G: \|\big(m
        (V) a(x \cdot f) - a_{xV}*f\big)(y)\| > 0\}$, we obtain
     \begin{align*}
    & \left\| m(Vx^{-1}) a(x\cdot f ) -
       \mathsmaller{\frac{m(Vx^{-1})}{m(V)}} (a_{xV}) \ast f
       \right\|_1\\ & = \Delta(x^{-1}) \int_G \left \| \big(m (V) a(x
       \cdot f) - a_{xV}*f\big) \right \| (y) dy\\ & = \Delta(x^{-1})
       \int_{E'} \left \| a f(x^{-1}y) \int_V ds - \int_G a_{xV}(ys)
       f(s^{-1}) ds \right \| dy \\ & \leq \Delta(x^{-1}) \|a\| \int_V
       \int_{E'} \left \| (x \cdot f)(y) -(xs \cdot f)(y)) \right \|
       \ dy ds \\ & \leq \Delta(x^{-1}) m(V) \sup_{s \in V} \left \|
       a\|\| f - s\cdot f \right \|_1, \quad \qquad (\text{since
         $\|x\cdot\varphi\|_1 = \|\varphi\|_1 \ \forall\, \varphi \in
         \A$}) \end{align*} where the second last inequality follows
     by left invariance and Tonelli's Theorem.

     Note that, by \Cref{cont}, for each $\varphi \in \A$, the maps $G
     \ni s \mapsto s\cdot \varphi, \varphi \cdot s \in \A$ are
     continuous. So, we can choose a compact symmetric neighbourhood
     $V_1$ of $e$ such that
\[
\sup_{s\in V_1}\norm{f-
  s\cdot f}_1  <
\frac{\epsilon}{ 2 \norm{a}};
\]
 and another compact symmetric neighbourhood $V_2$ of $e$ such that 
 \[
\sup_{s\in V_2}\norm{f\cdot x^{-1}- f \cdot (xs)^{-1}}_1 <
\frac{\epsilon}{ 2 \norm{a} \Delta(x^{-1})}.
\] 
  Using a compact  neighbourhood $V$ of $e$ satisfying $V
  \subseteq V_1 \cap V_2$ in $(*)$, we readily obtain
  \[
\| \Delta(x^{-1})((f\cdot
	x^{-1})a) - (a(x\cdot f )) - 
\mathsmaller{\frac{1}{m(V)}} ( f \ast (a_{xV}) - (a_{xV}) \ast f )
\|_1 \leq \epsilon
  \]
for all $x \in V$, as was desired.\smallskip
             
Conversely, suppose that $L$ is a closed subspace of $\A$ such that
	\[
        \Delta(x^{-1})(f\cdot x^{-1})a - a(x\cdot f ) \in L\ \text{ for
	all } f \in L, x \in G \text{ and } a \in A.
        \]
We show that $L$ is a Lie ideal.  Since $C_c(G, A)$ is dense in $\A$,
it is enough show that $\varphi \ast f - f \ast \varphi \in
\overline{L}$ for every $f \in L$ and $\varphi \in C_c(G, A)$.  Let $f
\in L$, $\varphi \in C_c(G,A)$ and $\epsilon > 0$. Let
$F:=\text{supp}(\varphi)$. Then, $K :=F \cup F^{-1}$ is a symmetric
compact set containing the support of $\varphi$.

As left multiplication, right multiplication, inverse function, the
function $\varphi$ and the modular function are all continuous, there
exists a mutually disjoint covering (not necessarily open) $K_1,
\dots, K_r$ of $K$ such that
\[
\norm{\varphi(y)(y\cdot f)- \Delta(y^{-1})(f\cdot y^{-1})\varphi(y)-
  \varphi(z)(z\cdot f) + \Delta(z^{-1})(f\cdot z^{-1})\varphi(z)}_1
\leq \epsilon/m(K) 
\]
 for all  $y,z \in K_i, 1 \leq i \leq r.$ Fix $y_i \in K_i $
for every $1 \leq i \leq r$ . Then,
	\begin{align*}
	& \norm{\varphi \ast f - f \ast \varphi -
            \mathsmaller{\sum}_{i} m(K_i) \big( \varphi(y_i)(y_i \cdot
            f) - \Delta(y_i^{-1}) ( f \cdot y_i^{-1})
            \varphi(y_i)\big) }_1 \\ &= \int_G \left \| (\varphi \ast
          f)(x) - \mathsmaller{\sum}_{i} m(K_i) \varphi(y_{i})
          (y_{i}\cdot f)(x) - (f \ast \varphi)(x)  \right. \\ & \qquad \qquad \left. + 
          \mathsmaller{\sum}_{i} m(K_i) \Delta(y_i^{-1}) ( f \cdot
          y_i^{-1})(x) \varphi(y_i)) \right \| dx \\ & = \int_G \big
          \| \mathsmaller{\int}_{G} \varphi(xy)f(y^{-1}) dy -
          \mathsmaller{\sum}_{i} \mathsmaller{\int}_{K_i}
          \varphi(y_{i})(y_{i} \cdot f)(x)dy - \mathsmaller{\int}_{G}
          f(xy)\varphi(y^{-1}) dy \\ & \qquad \qquad +
          \mathsmaller{\sum}_{i} \mathsmaller{\int}_{K_i}
          \Delta(y_i^{-1}) ( f \cdot y_i^{-1})(x) \varphi(y_i) dy \big
          \|dx\\ & = \int_G \big \| \mathsmaller{\int}_{K}
          \varphi(y)f(y^{-1}x) dy - \mathsmaller{\sum}_{i}
          \mathsmaller{\int}_{K_i} \varphi(y_{i})(y_{i} \cdot f)(x)dy
          - \mathsmaller{\int}_{K} \Delta(y^{-1})(f \cdot y^{-1})(x)
          \varphi(y) dy\\ & \qquad \qquad + \mathsmaller{\sum}_{i}
          \mathsmaller{\int}_{K_i} \Delta(y_i^{-1}) ( f \cdot
          y_i^{-1})(x) \varphi(y_i) dy \big \|dx \quad \text{\hspace*{50mm} (by
            \ref{f-inverse})}\\ & = \int_G \big \|
          \mathsmaller{\sum}_{i} \mathsmaller{\int}_{K_i}\varphi(y)(y
          \cdot f)(x)dy - \mathsmaller{\sum}_{i}
          \mathsmaller{\int}_{K_i} \varphi(y_{i})(y_{i} \cdot f)(x)dy
          - \mathsmaller{\sum}_{i}
          \mathsmaller{\int}_{K_i}\Delta(y^{-1})(f \cdot y^{-1})(x)
          \varphi(y) dy \\ & \qquad \qquad +
          \mathsmaller{\sum}_{i}
          \mathsmaller{\int_{K_i}} \Delta(y_i^{-1})(f \cdot
          y_i^{-1})(x) \varphi(y_i) dy \big \| dx\\ &
           \leq \sum_{i}
          \int_{K_i} \int_{G} \left \| \varphi(y)(y \cdot f)(x) -
          \varphi(y_{i})(y_{i} \cdot f)(x) - \Delta(y^{-1}) (f \cdot
          y^{-1})(x)\varphi(y)\right.\\ &\qquad \qquad \left. + \Delta(y_i^{-1})(f \cdot y_i^{-1})(x)
          \varphi(y_i) \right \|dx dy\\ & 
          = \sum_{i} 
          \int_{K_i} \left \| \varphi(y)(y \cdot f) -
          \varphi(y_{i})(y_{i} \cdot f) - \Delta(y^{-1}) (f \cdot
          y^{-1})\varphi(y) + \Delta(y_i^{-1})(f \cdot y_i^{-1})
          \varphi(y_i) \right \|_1 dy\\ \ 
          & \leq \epsilon,
	\end{align*}
where the third last inequality is obtained by appealing to Tonelli's
Theorem on $E \times K_i$, where $E$ is a $\sigma$-finite Borel
subset of $G$ similar to the one considered  (right after $(*)$) while proving necessity. This proves the
result.
\end{proof}

\begin{cor}\label{LieidealsinL1GC}
 Let $G$ be a locally compact group. Then, a
 closed subspace $L$ of $L^1(G)$ is a Lie ideal if and only if
  \[  \Delta(x^{-1})f\cdot x^{-1}- x\cdot f \in L
  \]
  for all $f \in L$
  and $x \in G$.
\end{cor}

\begin{remark} When the Haar measue is $\sigma$-finite (equivalently, $G$ is $\sigma$-compact), $L^{\infty}(G)$ is the
dual space of $L^1(G)$; so, sufficiency in \Cref{LieidealsinL1GC}
follows easily when $A = \C$. Indeed, for $f \in L$, $h \in L^1(G)$
and $\phi \in L^{\perp}$, we have
	\begin{align*}
\phi (f \ast h - h \ast f) & = \int_{G} \phi(x) (f \ast h - h \ast
f)(x) dx \\ 
&= \int_{G} \phi(x) \Big( \int_{G} f(xy) h(y^{-1})dy -
\int_{G} h(xy)f(y^{-1})dy \Big )dx\\ 
& = \int_{G} \phi(x) \Big(\int_{G} h(y^{-1})f(xy) dy - \int_{G} h(y)f(y^{-1}x)dy \Big )dx \\
&= \int_{G} \phi(x) \Big(\int_{G} h(y)f(xy^{-1})\Delta(y^{-1}) - h(y)f(y^{-1}x) \Big)dydx \\
&= \int_{G}h(y) \Big( \int_{G} \phi(x) (\Delta(y^{-1}) f\cdot y^{-1}- y\cdot f)(x)dx \Big)dy =0.
	\end{align*}
So, by Hahn-Banach Theorem, we deduce that $L$ is a closed Lie ideal in $L^1(G)$.
\end{remark}

Since $(G\times \{ e_H\}) \cdot (\{e_G\} \times H)= G \times H$, it is
easily seen that a subspace $I$ of $L^1(G \times H)$ is an ideal if
and only if $f\cdot x$ and $x\cdot f \in I$ for every $f \in I$ and $x
\in (G\times \{ e_H\}) \cup (\{e_G\} \times H)$.  However, these
elements are not enough to characterize the closed Lie ideals of
$L^1(G \times H)$, as we illustrate in the following example.

\begin{example}
	Consider $G = H=D_6$. Consider the subspace $L$ of $\C[G\times
          H]$ given by
        \[
        L = \{ c_1(r,e)-c_1(r^2,e)+ c_2(s,e) -c_2(r^2s,e) + c_3(rs,e)
        -c_3(r^2s,e) : c_1,c_2,c_3 \in \C \}.
        \]
         Since the modular function of a finite group is identically $1$, it
 is readily seen that $\Delta(x^{-1}) f\cdot x^{-1}- x\cdot f \in L$
 for every $f \in L$ and $x \in (D_6 \times \{ e\}) \cup (\{e\} \times
 D_6)$. However, $L$ is not a Lie ideal in $\C[G \times H]$, as after
 some routine calculation, we observe that
	\begin{align*}
&	 \Big( (r,e)- (r^2,e) + (s,e) - (r^2s,e) + (rs,e)  - (r^2s,e)
          \Big)  (r,r)\\
          &- (r,r)  \Big( (r,e) - (r^2,e) + (s,e) -
          (r^2s,e) + (rs,e) - (r^2s,e) \Big) \\
          &= -(rs,r) + (s,r)
          -(rs,r) - (rs,r) + (s,r) + (s,r) \\
          &= -3(rs,r)+ 3(s,r)\notin L. \hspace*{80mm}         \qed
	\end{align*}
\end{example}

However, when $H$ is abelian, then a smaller collection of $G
\times H$ with an extra condition provides sufficiency for a subspace to be a Lie ideal.

\begin{cor}\label{H-abelian}
	Let $G$ and $H$ be locally compact
        groups and $H$ be abelian. Let $A$ be a Banach algebra and $L$
        be a closed subspace of $L^1(G \times H, A)$.  Then, $L$ is a
        Lie ideal if
        \[
        w\cdot f \in L\ \text{and }\ \Delta(z^{-1})(f\cdot
        z^{-1})a- a(z\cdot f) \in L\] \(\text{ for all}\  w\in \{e_G\}\times H, f \in
        L, a \in A \text{ and}\ z \in G \times \{ e_H\} \).
\end{cor}

\begin{proof}
By \Cref{LieidealsinL1G}, it suffices to show that
$\Delta(x^{-1},y^{-1})\, \big(f\cdot (x^{-1},y^{-1})\big) a- a \big(
(x,y)\cdot f\big) \in L$ for all $ f\in L, a \in A$ and $(x, y)\in G
\times H$. Note that,  $\Delta_{|_{\{e_G\} \times H}}$ is the modular function of $H$ and
hence it is trivial on $H$. Thus, we have
\begin{eqnarray*}
\lefteqn{ \Delta((x^{-1},y^{-1}))\, \big(f\cdot (x^{-1},y^{-1})\big) a- a\,  \big((x,y)\cdot f\big)} \\ 
& = & \Delta((x^{-1},e_H))\Delta ((e_G,y^{-1}))  \big( f \cdot (x^{-1},e_H)(e_G,y^{-1})\big) a- a \big((x,e_H)(e_G,y)\cdot f\big) \\ 
& = & \Delta((x^{-1},e_H))\big( f\cdot (x^{-1},e_H) \cdot (e_G,y^{-1})\big) a - a \big ((x,e_H)\cdot f \cdot (e_G,y^{-1})\big) \\ 
& = & \Big(\Delta((x^{-1},e_H)) \big( f\cdot (x^{-1},e_H) \big) a - a \big( (x,e_H) \cdot f\big) \Big) \cdot (e_G,y^{-1}) \\ 
& = & (e_G,y)\cdot \Big( \Delta((x^{-1},e_H)) \big( f\cdot (x^{-1},e_H)\big) a - a \big( (x,e_H)\cdot f\big) \Big),
\end{eqnarray*}
 which, by hypothesis, belongs to $L$.
\end{proof}

\begin{remark}
  The first part of the hypothesis of \Cref{H-abelian} is not
  necessary. For instance, taking $G$ to be the trivial group, $H$ 
  abelian and $A=\C$, we observe that every subspace of $L^1(G \times H, A)$ is 
a Lie
  ideal as $ L^1(G \times H, A)  \cong L^1(H)$ is a commutative Banach algebra 
and if
  $L$ a subspace of $L^1(G \times H, A)$ which is not a left ideal, then $(e, 
y)\cdot
  f\notin L$ for some  $y \in H$ and $f\in L$.
\end{remark}

\subsection{A partial dictionary between the Lie ideals of $L^1(G)$ and those of  $L^1(G,A)$}\( \)

We now look for Lie ideals in $\A$ that can be obtained from those in
$A$.  Towards this end, for each subspace $F$ of $A$, we consider the
subspace
\[
\widetilde{F} := \{ f \in \A: f(x) \in F\ \text{for
  almost every}\ x \in G\}
\]
contained in $\A$.  Clearly, when $F$ is closed so is $\widetilde{F}$.
\begin{remark}\label{abelian-L-tilde}
Let $I$ be a closed subspace of $A$. Then, it is easily seen that $I$
is an ideal in $A$ if and only if $\widetilde{I}$ is an ideal in $\A$.

However, not every closed ideal in $\A$ is of the form
$\widetilde{I}$.  If $A$ is a simple Banach algebra and every closed
ideal of $\A$ is of the form $\tilde{I}$, then $\A$ will be simple as
$(0)$ and $A$ are the only choices for $I$ which yield
$\widetilde{(0)} = (0)$ and $\tilde{A} = \A$. But, in general, $\A$
need not be simple even if $A$ is so.  For instance, whenever $L^1(G)$
has a proper and non-trivial closed ideal, say, $J$, then $\ol {J \ot
  A}^{\gamma}$ is a proper (see \cite[page 2]{Guichardet}) and
non-trivial closed ideal of $\A$.
\end{remark}

In view of \Cref{abelian-L-tilde}), one could ask similar questions
for closed Lie ideals.  It turns out that the characterization
obtained in \Cref{LieidealsinL1G} helps us to answer these questions
appreciably well.

\begin{prop}\label{L-abelian}
Let $G$ be a locally compact group, $A$ be a Banach algebra and $L$ be
a closed subspace of $A$. If $\widetilde{L}$ is a Lie ideal in $\A$,
then $L$ is a Lie ideal in $A$. The converse holds when $G$ is
abelian.
\end{prop}
\begin{proof}
  Let $l\in L$ and $a\in A$. Then, for any fixed Borel set $E$ with
  finite positive measure, $l\chi_E \in \widetilde{L}$.  Taking $x=e $
  in \Cref{theorem-eqn}, we see that $l\chi_E a - al \chi_E\in
  \widetilde{L}$; so that $la\chi_E(y) - al\chi_E(y) \in L$
  a.e.. Hence, $l a - a l \in L$, which implies that $L$ is a Lie
  ideal in $L$.

When $G$ is abelian, then the converse follows easily using \Cref{LieidealsinL1G}.
\end{proof}

The converse of the preceding proposition is not true in general as is
clear from the following proposition, which also shows that the
commutativity of $G$ has a significant say in telling whether
$\widetilde{L}$ is a Lie ideal in $\A$ or not.

\begin{prop}\label{non-abelian-L-tilde}
Let $G$ be a non-abelian locally compact group, $A$ be a Banach
algebra and $L$ be a closed subspace of $A$. Then, $\widetilde{L}$ is a
Lie ideal in $\A$ if and only if $L$ is an ideal in $A$ if and only if
$\widetilde{L}$ is an ideal in $\A$.
\end{prop}

\begin{proof}
 In view of \Cref{abelian-L-tilde}, we only need to prove necessity in
 the first equivalence. Suppose that $\widetilde{L}$ is a Lie ideal in
 $\A$. By \Cref{L-abelian}, $L$ must be a Lie ideal in $A$. Now,
 suppose on contrary that $L$ is not an ideal in $A$. Then, there
 exist $a \in A$ and $l \in L$ such that $al \notin L$ (or, $la \notin
 L$). Also, $G$ being non-abelian, there exist $x, y \in G$ such that
 $xy \neq yx$.  Let $V_1$ and $V_2$ be disjoint compact neighbourhoods
 of $xy$ and $yx$ respectively.  Since $x^{-1}V_1$ and $V_2 x^{-1}$
 are neighbourhoods of $y$, so is $V = x^{-1}V_1 \cap V_2
 x^{-1}$. Then, taking $f =l\, \chi_{V_1} \in \widetilde{L}$ (or, $f
 =l\, \chi_{V_2}$), we observe that for every $z \in V$,
	\begin{eqnarray*}
	\big( \Delta(x^{-1})a (x^{-1}\cdot f)- (f\cdot x) a  \big) (z)
        &       = & \Delta(x^{-1})a f(xz) - f(zx)a\\
      &  = & \Delta(x^{-1})  al
        \chi_{V_1}(xz) - la \chi_{V_1}(zx)\\
        &  = & \Delta(x^{-1}) al \quad (\text{or}, -la)\\
	& \notin&  L,
        \end{eqnarray*}
which gives a contradiction by \Cref{LieidealsinL1G}, as $m(V) >0$.
\end{proof}
It now remains only to answer the question whether every closed Lie
ideal in $\A$ is of the form $\widetilde{L}$ or not. It turns out that
it is not the case.  Recall that a topological group $G$ is said to be
an ${\bf [IN]}$ group if $G$ contains a compact neighbourhood of
identity invariant under inner automorphisms. Clearly, every locally
compact abelian group is an ${\bf [IN]}$ group and every ${\bf [IN]}$
group is unimodular.

\begin{prop}\label{IN-gps}
Let $G$ be a non-trivial ${\bf [IN]}$ group and $A$ be a \BA \ with
non-trivial center.  Then, there exists a closed Lie ideal in $\A$
which is not of the form $\widetilde{L}$ for any closed Lie ideal $L$
in $A$.
\end{prop}
\begin{proof}
Let $0\neq a \in \mcal{Z}(A)$, $V$ be a compact invariant
neighbourhood of $e \in G$ and $f := a\chi_{V} \in \A$.  Then,
$(f\cdot x)a'-a'(x^{-1}\cdot f) =0$ for all $x \in G$ and $ a' \in A$.
Thus, by \Cref{LieidealsinL1G}, $\text{span}\{f\}$ is a closed Lie
ideal in $\A$.  We assert that $\text{span}\{f\}$ is not of the form
$\widetilde{L}$ for any closed Lie ideal $L$ in $A$.

Suppose, on contrary, that there exists a closed Lie ideal $L$ in $ A$
such that $\text{span}\{f\} = \widetilde{L}$. Clearly $a \in L$.  If
we can choose a Borel set $U $ of finite measure such that $m(U
{\bf \Delta} V) \neq 0$, then for $h:=a\chi_{U} \in \A$, we easily see
that $h \in \widetilde{L}$ whereas $h \notin \text{span}\{f\}$, which
gives a contradiction.

So, it just remains to show that such a choice of $U$ is possible. If
$V = \{ e\}$ (equivalently, $G$ is discrete), then for any $e \neq y
\in G$, $U$ can be taken as any compact neighbourhood of $y$ not
intersecting $V$. On the other hand if $V \neq \{e\}$, fix an $e \neq
x \in V$, and then take any two disjoint neighbourhoods $V_1$ and
$V_2$ of $e$ and $x$, respectively, inside $V$. One may take $U$ as $
V_1$ or $V_2$.
\end{proof}

\section{Center of a  generalized group algebra}

We first recall certain classes of topological groups that we will come
across while discussing center of generalized group algebras (in fact,
we already met one such class in \Cref{IN-gps}). Recall that, for any
topological group $G$,
\[
\mathrm{Aut}(G):=\{\beta: G \ra G: \beta \text{ is a bi-continuous group isomorphism}\}
\]
is a Hausdorff topological group (not necessarily locally compact)
with respect to the so called {\em Birkhoff topology}.  A basis for
the neighbourhoods of the identity automorphism $I$ in this topology
is given by the collection
\[
\{ N(K, V) : K \subseteq G \text{ is compact and }
V \subseteq G \text{ is a   neighbourhood of } e\},
  \]
   where \( N(K, V):= \{ \beta \in \mathrm{Aut}(G) : \beta(x),
   \beta^{-1}(x) \in Vx \text{ for all } x \in K\}.  \) See
   \cite{braconnier} and \cite{peters} for details. A topological
   group $G$ is said to be an
\begin{enumerate}
\item  ${\bf [IN]}$ group if $G$ contains a compact neighbourhood
  $U$ of the identity of $G$ which is invariant under inner automorphisms,
  i.e., $gUg^{-1} = U$ for all $g \in G$.
\item  ${\bf [SIN]}$ group if every neighbourhood of identity
  contains a compact neighbourhood of identity which is invariant
  under inner automorphisms.
  \item ${\bf [FIA]}^{-}$ group if $\mathrm{Inn}(G)$, the group of
    inner automorhphisms of $G$, is relatively compact in
    $\mathrm{Aut}(G)$ with respect to the Birkhoff topology.
    \item  ${\bf [FC]}^{-}$ group if the conjugacy classes of $G$
      are all relatively compact.
\end{enumerate}
A detailed analysis of above classes of topological groups is available in
\cite{grosser} (also see \cite{braconnier, peters, palmerpaper}).

\begin{remark}\label{types-of-groups}
  \begin{enumerate}
\item A topological group is an ${\bf [FIA]}^{-}$
  group if and only if it is an ${\bf [FC]}^{-}$  as well as an
  ${\bf [SIN]}$ group - see \cite{palmerpaper}.
\item  Birkhoff topology is known to be finer than the topology of uniform
 convergence on compacta. In particular, this implies that the
 evaluation map $\mathrm{Aut}(G) \times G \ra G$ is continuous  - see \cite{braconnier, peters}. 
  \end{enumerate}
\end{remark}

We now proceed to examine the elements of the center of a  generalized
group algebra.

In 1972, Mosak (\cite[Proposition 1.2]{mosak1}) proved  that for a
locally compact group $G$, the center of its group algebra is given by
\[
\mcal{Z}(L^1(G)) = \{ f \in L^1(G) :
\Delta(x^{-1})(f \cdot x^{-1}) = (x\cdot f), \ \forall \ x \in G\}.
\]
Motivated by this, we obtain a similar realization of the center of a
generalized group algebra by employing some of the techniques used in
\Cref{LieidealsinL1G}.  Note that $\mcal{Z}(\A)$ is a closed Lie
ideal in $\A$.
\begin{prop}\label{centre1}
	Let $G$ be a locally compact group and $A$ be
	a \BA. Then,
	\[
	\mcal{Z}(\A) = \{ f \in \A : \Delta(x^{-1})(f\cdot x^{-1})a =
	a(x \cdot f), \forall \ x \in G, a \in A \}.
	\]
\end{prop} 

\begin{proof}
	Let $f \in \mcal{Z}(\A), x\in G$ and $a \in A$, then as proved
        in \Cref{LieidealsinL1G}, for $\epsilon >0$ there exists a
        compact neighbourhood $V$ of $e$ such that
	\[
	\mathsmaller{\|} \Delta(x^{-1})(f\cdot x^{-1})a - a(x\cdot f ) -
	\mathsmaller{\frac{1}{m(V)}} ( f \ast (a_{xV}) - (a_{xV}) \ast f )
	\mathsmaller{\|}_1 \leq \epsilon.
	\]
	The required equality follows easily from the fact that $f
        \ast (a_{xV}) = (a_{xV}) \ast f $.  For the reverse inclusion,
        let $f \in \A$ be such that $\Delta(x^{-1})(f\cdot x^{-1})a =
        a(x\cdot f )$ for every $x \in G, \ a\in A$.  It is enough to
        show that $f \ast \varphi = \varphi \ast f$ for every $\varphi
        \in C_c(G,A)$. Let $\varphi \in C_c(G,A)$. Then, for given
        $\epsilon >0$, as proved in \Cref{LieidealsinL1G}, we have
	\[
	\mathsmaller{\|} \varphi \ast f - f \ast \varphi -
	\mathsmaller{\sum}_{i} m(K_i)(\varphi(y_i)(y_i \cdot f) -
	\Delta(y_i^{-1}) ( f \cdot y_i^{-1}) \varphi(y_i) )
	\mathsmaller{\|}_1 \leq \epsilon,
	\]
	where $K_i$ and $y_i$'s are as mentioned in the proof of
	\Cref{LieidealsinL1G}. Hence the result.
\end{proof}

It is also known that for any unimodular locally compact group $G$, $f
\in L^1(G)$ is central if and only if $f$ is constant on the conjugacy
classes of $G$ (see \cite{losert}).  However, such a characterization
does not hold for $\A$, in general. For example, if $G = D_6 :=\langle
r, s : r^3 = 1, s^2 = 1, srs = r^{-1}\rangle$ and $A$ is a Banach
algebra endowed with the trivial multiplication, i.e., $ab = 0$ for
every $a,b \in A$, then $\Z(L^1(D_6, A))= L^1(D_6, A)$ as $f \ast g =
0$ for every $f , g \in L^1(D_6, A)$.  And, one may define $f\in
L^1(D_6, A)$ such that $f(s) = c$ and is zero otherwise, for some
$0\neq c \in A$; then, $f(rsr^{-1}) = f(rsr^2) = f(r^2s) = 0 \neq c=
f(s) $. Thus, $f$ is central but not constant on conjugacy classes.

	Even the converse is not true in general, i.e., functions
        which are constant on the conjugacy classes need not be in
        $\mcal{Z}(\A)$.  For example, every function in $ L^1(\R,
        M_2)$ is constant on conjugacy classes. Define $f(x)
        = \begin{bmatrix} 1& 0\\ 0& 0 \end{bmatrix}$ when $x \in
        [0,1]$ and zero otherwise. Then,  for $x = 1$, $a
        = \begin{bmatrix} 1& 1\\ 0& 0 \end{bmatrix}$ and for any $y \in [1,2]$,
        we have
        \[
        (f \cdot x^{-1} )a (y) = \begin{bmatrix} 1& 1\\ 0&
          0 \end{bmatrix} \neq \begin{bmatrix} 1& 0\\ 0&
          0 \end{bmatrix} = a(x \cdot f) (y);
        \]
        thus, $f \notin \Z(L^1(\R, M_2))$, by \Cref{centre1}.\smallskip
        
       Interestingly, when $G$ is  ${\bf [SIN]}$  and $A$ is
       unital, we can have a satisfying identification of the elements
       of the $\Z(\A)$. We need some preparation for this.  Our first
       step in this direction is to prove that the elements of
       $\Z(\A)$ are all $\Z(A)$-valued.

\begin{lemma}\label{centervaluedfunctionsincenter}
Let $G$ be a locally compact group and $A$ be a
Banach algebra. Then,
\[
\Z(\A) \subseteq L^1(G, \Z(A)).
\]
\end{lemma}
\begin{proof}
Consider a non-zero element $f$  in $\mcal{Z}(\A)$.  Suppose there exists a
positive measure Borel set $E'$ in $G$ such that $f(x) \notin \Z(A)$ for every
$x \in E'$. Choose a measurable subset $E$ of $E'$ of finite and
positive measure.

By \cref{centre1} we obtain $f a = a f$ for every $a \in A$.  Let
$B_{E}$ denote the $\sigma$-algebra consisting of all Borel sets
contained in $E$.  Then, for any $F \in B_{E}$, we find that
\[
\left( \int_F f(x)\, dx\right)\, a = \int_F f(x)a\, dx = \int_F
(fa)(x) dx = \int_F (af)(x)\, dx = a\,\left(\int_F f(x)\, dx\right)
\]
for all $a\in A$, i.e., $\int_F f dx \in \Z(A)$ for all $F \in
B_E$.  Define $H: B_E \rightarrow \Z(A)$ by $H(F)
= \int_F f dx$.  Then, $H$ is an $m$-continuous (i.e., $\lim_{m(F) \to
  0}H(F) = 0$) vector measure of bounded variation (\cite[Theorem
  II.2.4]{diestel}). Thus, by \cite[Corollary III.2.5]{diestel}, there
exists a $g \in L^1(E,\Z(A))$ such that $H(F) = \int_F g(x) dx$ for
all $F \in B_E$.  This shows that $\int_F (f(x) - g(x)) dx =
0$ for every $F \in B_{E}$; so that, $f = g$ a.e. on $E$, by
\cite[Corollary II.2.5]{diestel}.  Since $g(E) \subseteq \Z(A)$, this
is a contradiction to the existence of $E'$.  Hence, $f(x) \in \Z(A)$
for almost every $x \in G$.
\end{proof}

Before looking at the elements of $\mcal{Z}(L^1(G,A))$, one needs to
worry when is it non-trivial. It is known that $\mcal{Z}(L^1(G))$ is non-trivial if
and only if $G$ is an ${\bf[IN]}$ group (\cite{mosak}). Adapting the
techniques of \cite{mosak}, we now provide its analogue for
generalized group algebras.
\begin{lemma} \label{centre}
	Let $G$ be a locally compact group and $A$ be
	a unital \BA. Then, $\mcal{Z}(\A) \neq (0)$ if and only if
	$G$ is an ${\bf[IN]}$ group.
\end{lemma}

\begin{proof}
	For $0 \neq f \in \mcal{Z}(\A)$, the function $h (x) = \| f(x)
	\|^{1/2} \in L^2(G)$ is positive. So, the function $p: G \ra \C$
	defined as $p(s) = \int_G h(sy)h(y) dy$ belongs to $C_0(G)$ (see
	\cite[Theorem 20.16]{hewittross}). Also, taking $a=1_A$ (the unit of
	$A$) in \cref{centre1}, one can easily verify that for every $t\in G$,
	$h(txt^{-1}) = \Delta(t)^{1/2}h(x)$, for almost every $x \in G$ (where
	the null set depends on $t$).
	Thus, $p$ is invariant
		under inner automorphisms and $p(e) = \|f \|_1 >0$. Then, for any
	$\epsilon$ such that $p(e) > \epsilon > 0$, the set $\{x \in G: p(x)
	\geq \epsilon \}$ is a compact neighbourhood of  $e$ in $G$, which is
	invariant under inner automorphisms. Hence, $G$ is an ${\bf[IN]}$-group.

 Conversely, if $G$ is an ${\bf [IN]}$ group, then
 $\mathcal{Z}\big(L^1(G)\big)$ is non-trivial and
\(
(0) \neq
\mathcal{Z}\big( L^1(G)\big) \otimes \C 1 \subseteq
\mathcal{Z}\big(L^1(G)\obp A \big).\)
\end{proof}

Before identifying the elements of $\mathcal{Z}(\A)$, for the sake of
clarity, in addition to the first paragraph of Section
2, we recall few more terminologies and facts related to Bochner
integrability of vector valued functions. Consider a measure space
$(\Omega, \mathcal{M}, \mu)$, a Banach space $X$ and a function $f:
\Omega \ra X$. Then, $f$ is said to be
\begin{enumerate}
\item Borel $(\mathcal{M},
\mu)$-measurable if $f^{-1}(U) \in \mathcal{M}$ for every open subset
$U$ of $X$.
\item $(\mathcal{M}, \mu)$-essentially separably valued if
there exists an $E \in \mathcal{M}$ such that $\mu(E^c) = 0$ and
$f(E)$ is contained in a separable (closed) subspace of $X$.
\end{enumerate}
A proof of the following useful equivalence can be
found, for instance, in \cite[Proposition 2.15]{ryan}.

\begin{prop}\label{equivalence}
  Let $(\Omega, \mathcal{M}, \mu)$ be a measure space, $X$ be a Banach
  space and $f: \Omega \ra X$ be a function supported on a
  $\sigma$-finite set.  Then, $f$ is $(\mathcal{M}, \mu)$-measurable
  if and only if $f$ is Borel $(\mathcal{M}, \mu)$-measurable and
  $(\mathcal{M}, \mu)$-essentially separably valued.
\end{prop}

Let $B_{\mathrm{inv}}:=\{B \in B_G: xBx^{-1} = B\ \text{for all } x \in
G\}$. Clearly, $B_{\mathrm{inv}}$ is a $\sigma$-subalgebra of $B_G$. Let
$L^1_{\mathrm{inv}}(G,A)$ denote the corresponding generalized group algebra
with repsect to the left Haar measure $m$ (as above).

\begin{lemma}\label{L-inv}
  Let $G$ be a locally compact group and $A$ be a Banach algebra.
\begin{enumerate}
\item If $f: G \ra A$ is a function with $\sigma$-finite support
  ({with respect to $B_{\mathrm{inv}}$}), then $f$ is $(B_{\mathrm{inv}},m)$-measurable
  if and only if $f$ is $(B_G, m)$-measurable and is constant on the
  conjugacy classes of $G$.
\item  $L^1_{\mathrm{inv}}(G,A) \subseteq \A$ as a closed subspace.  
\end{enumerate}
In particular,
\[
L^1_{\mathrm{inv}}(G,A) =\{f \in \A:  f\ \text{is constant on the conjugacy
  classes of $G$}\}.
\]
\end{lemma}

\begin{proof}
  (1) Let $f$ be $(B_{\mathrm{inv}},m)$-measurable.  Since
  $B_{\mathrm{inv}} \subseteq B_G$, $f$ is Borel $(B_G,
  m)$-measurable, by \Cref{equivalence}.  Also, since $f$ is
  $(B_{\mathrm{inv}}, m)$-essentially separably valued, there exists
  $E \in B_{\mathrm{inv}} \subseteq B_G$ such that $m(E^c) = 0$ and
  $f(E)$ is contained in a separable subspace. Hence, $f$ is $(B_G,
  m)$-essentially separably valued.  To see that $f$ is constant on
  conjugacy classes, consider $x, t\in G$. Since $f^{-1}(\{f(t)\})\in
  B_{\mathrm{inv}}$, we have $xf^{-1}(\{f(t)\})x^{-1} =
  f^{-1}(\{f(t)\})$ so that $f(xtx^{-1}) = f(t)$. Since $x$ and $t$
  were arbitrary, it follows that $f$ is constant on the conjugacy
  classes of $G$.
  
  Conversely, suppose $f$ is $(B_G,m)$-measurable and is constant on
  conjugacy classes. Then, by \cref{equivalence} again, $f$ is Borel
  $(B_G,m)$-measurable and $(B_G, m)$-essentially separably valued. We
  first show that $f$ is Borel $(B_{\mathrm{inv}},m)$-measurable.  Let $x \in
  G$ and $U \subset G$ be open. Then,
  \[
x^{-1} f^{-1}(U)x = \{ x^{-1} g x : f(g) \in U \} = \{ x^{-1} x g
x^{-1} x : f(g) \in U \} = f^{-1} (U)
  \]
because $f(g) = f(xgx^{-1})$ for every $g \in G$.  This proves that
  $f^{-1}(U) \in B_{\mathrm{inv}}$ and hence $f$ is Borel
  $(B_{\mathrm{inv}},m)$-measurable.

Further, since $f$ is $(B_G, m)$-essentially separably valued, there
exists an $E\in B_G$ such that $m(E^c) = 0$ and $f(E)$ is contained in
a separable closed subspace $Y$ of $A$.  Let $\tilde{E} =
f^{-1}(Y)$. Since $Y^c$ is open and $f$ is Borel
$(B_{\mathrm{inv}},m)$-measurable, $\tilde{E} \in B_{\mathrm{inv}}$. Clearly,
$f(\tilde{E}) \subset Y$ and since $f^{-1}(Y^c) \subseteq E^c$, we
have $m(\tilde{E}^c) = m(f^{-1}(Y^c)) = 0$.  Hence, $f$ is $(B_{\mathrm{inv}},
m)$-essentially separably valued.

In particular, by \cref{equivalence} again, $f$ is
$(B_{\mathrm{inv}},m)$-measurable. \smallskip
	
(2) From the definition of Bochner integrability, it is
clear that $L^1_{\mathrm{inv}}(G,A) \subseteq \A$. And, since
\[
\|f\|_{L^1_{\mathrm{inv}}(G, A)} = \big\| \|f\|
\big\|_{L^1_{\mathrm{inv}}(G)} = \big\|
\|f\|\big\|_{L^1(G)} = \|f\|_{\A}
\]
for every $f\in L^1_{\mathrm{inv}}(G,A)$, it follows that $L^1_{\mathrm{inv}}(G,A)$ is closed in $\A$.
\end{proof}

\begin{theorem}\label{constantonconjugacy}
	Let $G$ be a locally compact ${\bf [IN]}$ group and $A$ be a
        unital Banach algebra. If either $G$ is an ${\bf [SIN]}$ group 
        or $\Z(A) =\C1_A $, then
	\begin{equation*}
	\mcal{Z}(\A) = \{f\in L^1(G, \mcal{Z}(A)): f \ \text{is
		constant on the conjugacy classes of $G$} \}.
	\end{equation*}
	\end{theorem}

 \begin{proof}
 Suppose $f\in L^1(G, \mcal{Z}(A))$ and is constant on conjugacy
 classes.  Then, for every $x\in G$ and $a \in A$, we have
	 $$ \Big(x^{-1} \cdot ((f \cdot x^{-1}) a) - x^{-1} \cdot (a (x
	 \cdot f))\Big)(y) = f(x y x^{-1}) a - a f(y) = f(y) a- f(y)a =
	 0,$$
	 for every $y \in G$.  Thus, $(f \cdot x^{-1}) a = a (x \cdot
         f) $ for all $x \in G$ and $a \in A$, which by \Cref{centre1}
         shows that $f \in \Z(\A)$, since $G$ is unimodular.
	 
 Conversely, consider any $ f \in \mcal{Z}(\A)$. By
 \Cref{centervaluedfunctionsincenter}, $f$ is $\mathcal{Z}(A)$-valued.

 If $\Z(A)= \C1_A$, then $f \in L^1(G, \C 1_A) \cong L^1(G)$.  In
 particular, $f$ commutes with every element of $L^1(G, \C 1_A)
 \cong L^1(G)$ and hence  is constant on the conjugacy classes of
 $G$, by \cite{losert}.
	 
 We next consider the case when $G$ is an ${\bf [SIN]}$ group. Then,
 there exists an approximate identity $\{ u_{\alpha} \}$ of $L^1(G)$
 contained in $L^{\infty}(G) \cap \Z(L^1(G))$ determined by a family
 of characteristic functions of compact invariant neighbourhoods of identity.
 Then, $\{u_{\alpha} \ot 1_A \}$ is an approximate identity of $L^1(G)
 \obp A$ contained in $\Z(L^1(G) \obp A)$. Thus, $\{f_{\alpha} = 1_A
 u_{\alpha}\}$ is an approximate identity of $\Z(\A)$ contained in
 $L^{\infty}(G,A)$. 

Then, $f_{\alpha} \ast f \in \Z(\A)$ because $\Z(\A)$ is a subalgebra
of $\A$.  Note that, for any $s,t \in G$, we have
\begin{eqnarray*} 
 \| f_{\alpha} \ast f(s) - f_{\alpha} \ast f(t) \| &=& \left \| \int_G
 f_{\alpha} (sx) f(x^{-1}) dx - \int_G f_{\alpha} (tx) f(x^{-1}) dx
 \right \| \\ & =& \left \| \int_G f_{\alpha}(x^{-1}) (f \cdot s)(x)
 dx - \int_G f_{\alpha}(x^{-1}) (f \cdot t)(x) dx \right \| \\ &=&
 \left \| \int_G \big( f_{\alpha}(x^{-1}) (f \cdot s)(x) -
 f_{\alpha}(x^{-1}) (f \cdot t)(x)\big) dx \right \| \\ & \leq &
 \int_G \| f_{\alpha}(x^{-1}) (f \cdot s)(x) - f_{\alpha}(x^{-1}) (f
 \cdot t)(x) \| dx \\ & \leq & \int_G \| ((f \cdot s)- (f \cdot t))(x)
 \| dx \qquad \qquad (\text{ since } \|f_{\alpha}\|_{\infty}=1\ \forall \alpha)\\ & = & \| f \cdot s - f \cdot t \|_1.
\end{eqnarray*}
 From \Cref{cont} it then follows that $f_{\alpha} \ast f$ is
 continuous.

Since $f_{\alpha} \ast f \in \Z(\A)$ and $A$ is unital, it follows
from \Cref{centre1} that $x^{-1} \cdot (f_{\alpha} \ast f) =
(f_{\alpha} \ast f) \cdot x$ for every $x \in G$.  Then, by continuity
of $f_{\alpha} \ast f$, we get $f_{\alpha} \ast f(xy) = f_{\alpha}
\ast f(yx)$ for every $y \in G$. Replacing $y$ by $x^{-1}y$ we see
that $f_{\alpha} \ast f \in L^1_{\mathrm{inv}}(G,A)$ for all
$\alpha$. By \Cref{L-inv}, $L^1_{\mathrm{inv}}(G,A)$ is closed in $\A$;
hence, $f$, being the  limit of the net $\{f_\alpha * f\}$, belongs to $L^1_{\mathrm{inv}}(G,A)$.
This proves that $f$ is constant on the conjugacy classes of $G$.
    \end{proof}

It was recently shown (in \cite[Theorems 1 $\&$ 2]{GJ2}) that for
$C^*$-subalgebras $A_0 \subseteq A$ and $B_0 \subseteq B$, $A_0 \obp
B_0$ can be identified with the Banach subalgebra $\overline{A_0
  \otimes B_0}^\gamma$ of $A \obp B$; and also that $\mathcal{Z}(A
\obp B) = \mathcal{Z}(A) \obp \mathcal{Z}(B)$. However, in general,
such identifications are not known for Banach algebras. For
generalized group algebras, we obtain such an identification of the
center for certain classes of groups. We need some auxiliary results
to prove this.

First, two lemmas that will be needed ahead, the first of which is
borrowed from \cite{braconnier} (also see \cite[Theorem 2.3 and Page
  8]{grosser} and \cite[$\S 1.0$]{mosak1}).

\begin{lemma}\cite{braconnier, grosser, mosak1}\label{mosak-lemma}
  Let $G$ be a locally compact group. Then, there exists a continuous
  homomorphism $\tilde{\Delta}: \mathrm{Aut}(G) \ra \R_{>0}$ such that
  \[
\int_G (f \circ \beta^{-1}) (x) dx = \tilde{\Delta} (\beta)\int_G f(x) dx
\] 
for every $f \in L^1(G)$ and $\beta \in \mathrm{Aut}(G)$.

Moreover,  $G$ is unimodular if and only if  $\tilde{\Delta} (\beta) = 1$ for
all $\beta \in \ol{\mathrm{Inn}(G)}$.
\end{lemma}

Based on this, we deduce the following: 
\begin{lemma}\label{beta-f-lemma}
Let $G$ be a  unimodular locally compact group and $A$ be a 
Banach algebra. Then,
  \begin{enumerate}
\item $\beta$ is measure preserving for every $\beta \in \ol{\mathrm{Inn}(G)}$; 
  \item $\beta \cdot f := f \circ \beta^{-1} \in \A$ and $\|\beta
    \cdot f \|_1 = \|f\|_1$ for every $\beta \in \ol{\mathrm{Inn}(G)}$
    and $f \in \A$; and
  \item the mapping
    \[
 \ol{\mathrm{Inn}(G)} \ni \beta \mapsto \beta \cdot f \in \A
    \]
is continuous for every $f \in \A$.
  \end{enumerate}
  \end{lemma}

 \begin{proof}
(1): Let $\beta \in \ol{\mathrm{Inn}(G)}$. Since $G$ is unimodular,
   for any $S \in B_G$, by \Cref{mosak-lemma}, we have
   \[
m(S)=  \tilde{\Delta}(\beta) m(S) =
   \tilde{\Delta} (\beta) \| \chi_S \|_1 = \| \chi_S \circ \beta^{-1}\|_1 = \| \chi_{\beta(S)}  \|_1
   =  m(\beta(S)).
  \] 

  (2): Let $\beta \in \ol{\mathrm{Inn}(G)}$ and $f \in \A$. We first
   show that $\beta\cdot f$ is $(B_G,m)$-measurable.
  
Since $f$ is $(B_G, m)$-measurable, there exists a sequence $\{s_n\}$ of $A$-valued simple measurable functions
and a null set $V$ such that $s_n\ra f$ on $V^c$. Then, by (1),
$\beta(V)$ is a null set; and, clearly, $\{\beta \cdot s_n\}$ is a
sequence of $A$-valued simple $(B_G, m)$-measurable functions
converging to $\beta \cdot f$ on $\beta( V)^c$. This proves that
$\beta \cdot f$ is $(B_G, m)$-measurable. For integrability, note that
$ \| \beta \cdot f \| = \|f\|\circ \beta^{-1} \in L^1(G)$, by
\Cref{mosak-lemma}.  Hence $\beta \cdot f \in \A$.

Next, by \Cref{mosak-lemma} again, we have
\[
\| \gamma \cdot f \|_1 = \int_G \| \gamma \cdot f \| = \int_G\|f
\|\circ \gamma^{-1} = \tilde{\Delta}(\gamma) \int_G \| f \|
=\tilde{\Delta}(\gamma) \| f \|_1
\]
for all $\gamma \in \mathrm{Aut}(G)$. Note that, $G$ being unimodular,
we have $\tilde{\Delta}(\gamma) = 1$ for every $\gamma \in
\ol{\mathrm{Inn}(G)}$ (by \Cref{mosak-lemma});  so that $\| \gamma \cdot f \|_1 =
\|f\|_1$ for all $\gamma \in \ol{\mathrm{Inn}(G)}$.
\smallskip

(3): Let $f \in \A$, $\{\beta_{\alpha}\}$ be a net converging to some
$\beta$ in $\ol{\mathrm{Inn}(G)}$ and $\epsilon >0$. Fix an  $h =
\sum_{i=1}^{n} f_i \ot a_i\in L^1(G) \ot A$ such that $\| f -
h \|_1 < \epsilon$.  For each $\alpha$, we have 
\[
\|\beta_{\alpha} \cdot f - \beta \cdot f \|_1 \leq \|
\beta_{\alpha} \cdot f - \beta_{\alpha} \cdot h 
\|_1 + \| \beta_{\alpha} \cdot h - \beta \cdot
h \|_1 + \| \beta \cdot h
- \beta \cdot f \|_1.
\]
On the right hand side of this inequality, the first and the third
terms are less than $\epsilon$ for every $\alpha$, becasue of (2); and
the middle term simplifies to
\[
\| \beta_{\alpha} \cdot \sum_{i=1}^{n} f_i a_i -
\beta \cdot \sum_{i=1}^{n} f_i a_i \|_1 \leq \sum_{i=1}^{n} \|
\beta_{\alpha} \cdot f_i - \beta \cdot f_i \|_1 \|a_i \|\ \text{ for all } \alpha.
\]
Now, using the fact that for each $g \in L^1(G)$, the mapping
$\ol{\mathrm{Inn}(G)} \ni \beta \mapsto \beta \cdot g\in L^1(G)$ is
continuous (see, \cite[Page no. 281]{mosak1}), we can choose an
$\alpha_0$ such that the last inequality is less than $\epsilon$ for
all $\alpha \geq \alpha_0$, and we are done.
 \end{proof}
\color{black}

 Prior to proving the main result, we also introduce a
 $\sharp$-operator on $L^1(G,A)$ with some useful properties. Note
 that such a $\sharp$-operator on $L^1(G)$ has been studied and used by
 various authors in the past (see \cite{mosak1} and the references
 therein). 
\begin{prop}\label{hash}
	Let $G$ be a locally compact ${\bf [FIA]}^-$ group and $A$ be a unital 
	Banach algebra. Then, there exists a projection
	$\sharp: L^1(G,A) \to L^1_{\mathrm{inv}}(G,A)$ which 
maps $ L^1(G,\Z(A))$ onto $ \Z(L^1(G,A)) $.
\end{prop}

\begin{proof}

	Since $G$ is an ${\bf [FIA]}^-$group, $\overline{\mathrm{Inn}(G)}$ is a
	compact subgroup of $\mathrm{Aut}(G)$ with respect to the
	Birkhoff topology; so, $\overline{\mathrm{Inn}(G)}$ has a unique
	normalized Haar measure, say, $d\beta$.
	\smallskip
	
	{\bf Step I:} We find a linear contraction $\sharp : \A \ra
        L^1_{\mathrm{inv}}(G,A)$.

        Note that, for $f \in C_c(G,A)$ and $x \in G$, since the
        evaluation map $\mathrm{Aut}(G) \times G \to G$ is continuous
        (\Cref{types-of-groups}), we observe that the mapping
	\[
	\overline{\mathrm{Inn}(G)} \ni \beta \mapsto (\beta \cdot f)(x) \in
	A
	\]
	is continuous. In particular, the map $\beta \mapsto \| (\beta
        \cdot f)(x)\|$, being continuous on a compact set, is
        integrable. This allows us to define a function $f^\sharp : G
        \ra A$ by
	\[
	f^\sharp(x) = \int\limits_{\overline{\mathrm{Inn}(G)}} (\beta
        \cdot f)(x) d\beta \text{ for } x \in G.
	\]
	We assert that $f^\sharp \in C_c(G, A)\cap
        L^1_{\mathrm{inv}}(G,A)$. Note that if $F$ is the compact
        support of $f$, then for any $y \notin
        \overline{\mathrm{Inn}(G)}(F)$, we have $f^\sharp(y) = 0 $,
        since $\beta^{-1}(y) \notin F$ for any $\beta \in
        \overline{\mathrm{Inn}(G)}$; thus, $\text{supp}(f^\sharp)
        \subseteq \overline{\mathrm{Inn}(G)}(F)$ which is compact
        (being the image of the compact set
        $\ol{\mathrm{Inn}(G)}\times F$ under the continuous evaluation
        map $ (\beta , x) \mapsto \beta(x)$). Thus, $f^\sharp$ is
        compactly supported.
	
	We now show that $f^\sharp$ is continuous.  Note that, being
        continuous and supported on a compact set, $f$ is left and
        right uniformly continuous; so, there exists a neighbourhood
        $V$ of $e$ such that $\|f(x) - f(y)\| < \epsilon$ whenever
        $x^{-1}y, y x^{-1} \in V$, by \cite[Theorem 4.15]{hewittross}.
        Further, since $G$ is an ${\bf [SIN]}$ group
        (\Cref{types-of-groups}), we can choose $V$ to be invariant.
       Then, for any $x,y \in G$ such that $x^{-1}y \in V$, it is
        easily seen that $(\beta^{-1}(x))^{-1}\beta^{-1}(y) \in V$ for
        every $\beta \in \mathrm{Inn}(G)$; so that $\|
        f(\beta^{-1}(x)) - f(\beta^{-1}(y)) \| \leq \epsilon$ for
        every $\beta \in \overline{\mathrm{Inn}(G)}$ and, hence,
	\[
	\| f^\sharp(x) - f^\sharp(y) \| \leq
	\int\limits_{\overline{\mathrm{Inn}(G)}} \| (f(\beta^{-1} x) -
	f(\beta^{-1} y)) \| d\beta \leq \epsilon.
	\]
        This proves that $f^\sharp$ is (uniformly) continuous. Also,
        note that, for any $x,y \in G$,
\begin{eqnarray*}
	f^\sharp(yxy^{-1}) & = &
	f^\sharp(\mathrm{Ad}_y^{-1}(x)) \\
        & = &  
	\int\limits_{\overline{\mathrm{Inn}(G)}} (\beta \cdot
	f)(\mathrm{Ad}_y^{-1} (x)) d\beta\\ & =
	&\int\limits_{\overline{\mathrm{Inn}(G)}} \big(\mathrm{Ad}_y
	\cdot(\beta \cdot f)\big)(x) d\beta \\ & = &
	\int\limits_{\overline{\mathrm{Inn}(G)}} (\beta\cdot f)(x)
	d\beta \qquad (\text{by left
		invariance of the Haar measure $d\beta$})
	\\
	& = & f^\sharp(x).\end{eqnarray*}
	Hence,  $f^\sharp \in L^1_{\mathrm{inv}}(G,A)$, by \Cref{L-inv}.

	We next prove that the map $\sharp: C_c(G,A) \ra
        L^1_{\mathrm{inv}}(G,A)$ is a linear contraction. Clearly,
        $\sharp$ is a linear map.  And, for any $f \in C_c(G,A)$, by
        \Cref{beta-f-lemma} and an appropriate application of
        Tonelli's Theorem, we have
	\[
	\| f^\sharp \|_1 = \int\limits_G \Big\|
        \int\limits_{\overline{\mathrm{Inn}(G)}} (\beta \cdot f)(x)
        d\beta \Big\| dx \leq \int\limits_{\overline{\mathrm{Inn}(G)}}
        \int\limits_G \|(\beta \cdot f)(x) \| dx d\beta =
        \int\limits_{\overline{\mathrm{Inn}(G)}} \|\beta \cdot f \|_1
        d\beta = \| f \|_1.
	\]	
Thus, $\sharp: C_c(G, A) \ra L^1_{\mathrm{inv}}(G,A)$ extends to a
linear contraction $\sharp: \A \to L^1_{\mathrm{inv}}(G,A)$.\smallskip
	
	{\bf Step II:} We now show that the operator $\sharp$ is
        identity on $L^1_{\mathrm{inv}}(G,A)$.

	Let $f \in L^1_{\mathrm{inv}}(G,A)$ and $\epsilon >0$. It
        suffices to show that $\| f^\sharp - f \|_1 \leq \epsilon$.
        For this, we first assert that $\beta \cdot f = f$ for every
        $\beta \in \overline{\mathrm{Inn}(G)}$. Note that, for $\beta
        = \mathrm{Ad}_y \in \mathrm{Inn}(G)$ for some $y \in G$, by
        \Cref{L-inv}, we have
	\[
	(\beta \cdot f)(x) = f(y x y^{-1})  = f(x) \text{ for all }  x \in G.
	\]
	 So, $\beta \cdot f = f$ for every $\beta \in
	\mathrm{Inn}(G)$; and, by the continuity of the map $\beta \mapsto
	\beta \cdot f$ (\Cref{beta-f-lemma}), we deduce that 
	\begin{equation}\label{beta-f-f}
	\beta \cdot f
	= f \text{ for every } \beta \in \overline{\mathrm{Inn}(G)}.
	\end{equation}
	
	Now, fix an $h \in C_c(G,A)$ such that $\| f - h \|_1 <
	\epsilon/3$. Then, imitating the proof  
	\cite[Lemma 1.4]{mosak1} verbatim, we see that $h^\sharp$ belongs to
	the closed convex hull of $\{\beta \cdot h: \beta \in
	\overline{\mathrm{Inn}(G)}\}$. So, there exist
	non-negative scalars $c_i, 1 \leq i \leq r$ with $\sum_{i=1}^r c_i =
	1$ and automorphisms $\{\beta_i : 1 \leq i \leq r\} \subset
	\overline{\mathrm{Inn}(G)}$ such that $\| h^\sharp - \sum_{i=1}^r c_i
	(\beta_i \cdot h) \|_1 \leq \epsilon/3$. Thus,
	\begin{align*}
	\| f^\sharp - f \|_1 & = \Big \| f^\sharp - \sum_{i=1}^r c_i
        (\beta_i \cdot f) \Big \|_1 \quad \text{\hspace*{50mm} (by
          Equation}\ (\ref{beta-f-f})) \\ & \leq \Big \| f^\sharp -
        h^\sharp \Big \|_1 + \Big \| h^\sharp - \sum_{i=1}^r c_i
        (\beta_i \cdot h) \Big \|_1 + \Big \| \sum_{i=1}^r c_i
        (\beta_i \cdot h) - \sum c_i (\beta_i \cdot f) \Big \|_1 \\ &
        \leq \Big \|f - h \Big \|_1 + \Big \| h^\sharp - \sum_{i=1}^r
        c_i (\beta_i \cdot h) \Big \|_1 + \sum_{i=1}^r c_i \Big \|h-f
        \Big \|_1 \qquad \qquad (\text{by Lemma }\ref{beta-f-lemma})
        \\ & \leq \epsilon.
	\end{align*}
	
	{\bf Step III:}	Finally, we show that $\sharp$ maps $L^1(G, \Z(A))$ onto
	$\Z(\A)$.

	Let $f \in L^1(G,\Z(A))$. Consider a sequence $\{f_n\} \subset
        C_c(G,\Z(A))$ such that $f_n \to f$, then $f_n^\sharp \to
        f^\sharp$. By the definition of the $\sharp$ operator,
        $f_n^\sharp \in C_c(G, \Z(A))\subset L^1(G, \Z(A))$, so that
        $f^\sharp \in L^1(G, \Z(A))$. The result now follows from the
        fact that $\Z(\A) = L^1(G, \Z(A)) \cap
        L^1_{\mathrm{inv}}(G,A)$ (by \Cref{constantonconjugacy}).
	\end{proof}

\begin{lemma}\label{centreofL1GA-FC}
	Let $A$ be a Banach algebra and $G$ be a locally compact group such that
	$\mathcal{Z}(L^1(G)) $ is complemented in $L^1(G)$ by a
	projection of norm one.  Then,
	\[
	\mathcal{Z}(L^1(G)) \obp \mcal{Z}(A) \subseteq \mathcal{Z}(L^1(G)
	\obp A).
	\]
\end{lemma}
\begin{proof}
	Since $\mathcal{Z}(L^1(G))$ is complemented in $L^1(G)$ by a
	projection of norm one, appealing to
	\cite[Proposition 2.4]{ryan} and \cite[Page
	30]{ryan}, the algebraic embeddings
	$$ \mcal{Z}\big(L^1(G)\big) \ot \mcal{Z}(A) \subseteq L^1(G) \ot
	\mcal{Z}(A) \subseteq L^1(G) \obp A $$ extend to isometric
	embeddings $$ \mcal{Z}\big(L^1(G)\big) \obp \mcal{Z}(A) \subseteq
	L^1(G) \obp \mcal{Z}(A) \subseteq L^1(G) \obp A .$$ Thus, it
	suffices to show that $\mathcal{Z}(L^1(G)) \ot \mcal{Z}(A) \subseteq
	\mathcal{Z}(L^1(G) \obp A)$. This follows easily by observing that
	$(f \ot a)(g \ot b) = fg \ot ab = gf \ot ba = (g \ot b)(f \ot a)$
	for every $f \in \mathcal{Z}(L^1(G))$, $a \in \mcal{Z}(A)$, $g \in
	L^1(G)$ and $b \in A$.
\end{proof}

\begin{remark} For every ${\bf [FC]}^-$ group $G$, $\mathcal{Z}(L^1(G))$ is
complemented in $L^1(G)$ by a projection of norm one (see
\cite{wells}). So, the preceding result holds for a large class of groups.
\end{remark}

We have equality in \Cref{centreofL1GA-FC} for some classes of
groups as we demonstrate below. For instance, when $G$ is abelian,
then $G$ is an $[FC]^-$ group and $\Z(L^1(G)) = L^1(G)$.  This gives
\[
L^1(G)
\obp \Z(A) = \Z(L^1(G)) \obp \Z(A) \subseteq \Z(L^1(G) \obp A)
\subseteq L^1(G) \obp \Z(A),
\]
where the last inclusion follows from
\Cref{centervaluedfunctionsincenter}. Hence, $\Z(L^1(G) \obp A) =
\Z(L^1(G)) \obp \Z(A)$.  And, more generally, we have the following:

\begin{theorem}\label{centerdistribution}
 Let $G$ be a locally compact group and $A$ be a unital Banach
 algebra. If either
 \begin{enumerate}
\item $\mathcal{Z}(L^1(G)) $ is complemented in $L^1(G)$ by a
  projection of norm one with either $G$ discrete or $\Z(A) = \C1_A$,
  or,
\item $G$ is an ${\bf [FIA]}^-$ group, then
 \end{enumerate}
\[
\mathcal{Z}(L^1(G)) \obp \mcal{Z}(A) = \mathcal{Z}(L^1(G) \obp
        A).
\]
\end{theorem}
\begin{proof}
  The direct inclusion follows from \Cref{centreofL1GA-FC}.  We prove
  the reverse inclusion in both the cases. Consider a non-zero element
  $f \in \mcal{Z}\big(L^1(G) \obp A\big)$. Note that, by
  \Cref{centervaluedfunctionsincenter}, $f\in L^1(G,\mcal{Z}(A))$.
	\smallskip

\noindent (1) Let $G$ be such that $\mathcal{Z}(L^1(G)) $ is
complemented in $L^1(G)$ by a projection of norm one.

First suppose that $\Z(A) = \C1_A$. Note that, from
\Cref{constantonconjugacy}, $f$ is constant on the conjugacy classes
of $G$ and takes values in $\Z(A)=\C 1_A$. So, there exists a $g \in
L^1(G)$ such that $f = g 1_A$.  Note that $g$ will also be constant on
the conjugacy classes of $G$; so that $g \in \Z(L^1(G))$ (by
\cite{losert}).  This proves that $f = g \ot 1_A \in \Z(L^1(G)) \obp
\C 1_A$ which gives the result.

Next suppose that $G$ is discrete. Then, $G$ is an ${\bf [SIN]}$ group
and from \Cref{constantonconjugacy}, it follows that $f$ is constant
on the conjugacy classes. Let $\{a_j: j \in \Gamma\}\subseteq
\mcal{Z}(A)$ denote the set of all possible values taken by $f$ on the
conjugacy classes of $G$, for some indexing set $\Gamma$.  For each
$j\in \Gamma$, let $C_{j}$ denote the union of all those conjugacy
classes on which $f$ takes the value $a_j$, and let $C:= G \setminus
\cup\{C_j : a_j \neq 0\}$; so that $f(x) = 0$ for all $x \in C$ if $C
\neq \emptyset$.  We assert that $\Gamma$ is countable.
	
For each $n\in \mathbb{N}$, let $\Gamma_n = \{j \in \Gamma: m(C_j)\,
\|a_j\| >1/n\}$. Since $G$ is discrete, it is easily seen that $\Gamma
\setminus \cup_{n=1}^{\infty} \Gamma_n$ is  at most a singleton; so, it
suffices to show that $\Gamma_n$ is finite for all $n$. Suppose, on
contrary, that for some $n \in \N$, $\Gamma_n$ contains countably infinite
members, say, $\{j_k\}_{k\in \mathbb{N}}$. Then,
\[
\int_G \| f\| \ dx \geq \int\limits_{\cup_{k=1}^{\infty}C_{j_k}} \| f\| \ dx =
\sum_{j=1}^{\infty} \int_{C_{j_k}} \|f\| \ dx = \sum_{k=1}^{\infty}
m(C_{j_k}) \|a_{j_k}\| \geq \sum_{k=1}^{\infty} \frac{1}{n} = \infty,
\]
which is a contradiction to the fact that $f$ is integrable, thereby
establishing our assertion.
	
Two possibilities arise, namely, either $\Gamma$ is finite or
infinite. If $\Gamma$ is finite, then $f= \sum_{j\in \Gamma}
\chi_{C_j}\otimes a_j\in \mathcal{Z}(L^1(G)) \otimes \mathcal{Z}(A)$. And, if $\Gamma$ is infinite, we can consider $\Gamma \setminus\{j: a_j=
0\}$ to be same as $\N$. Note that, $\|f - \sum_{j=1}^n \chi_{C_j}
a_j\|\leq \|f\|$ (as scalar functions) for all $n \geq 1$  and $\|f
- \sum_{j=1}^n \chi_{C_j} a_j\|\ra 0$ pointwise on $G$. Thus,  $ f =
\sum_{j=1}^\infty \chi_{C_j} a_j$ in $\A$, by Lebesgue Dominated
Convergence Theorem. Also, $ \chi_{C_j} a_j$ corresponds to $
\chi_{C_j}  \ot a_j$ which is in $
\mcal{Z}(L^1(G)) \obp \mcal{Z}(A)\subseteq \mathcal{Z}(L^1(G)
\obp A)$ for all $j \in \Gamma$. Hence the result.

%Since $\mathcal{Z}(L^1(G)) \obp \mcal{Z}(A) $, we obtain \begin{equation*}
 % f = \sum_{j=0}^{\infty}
%\|\chi_{C_j} f\| \ot \frac{a_j}{\| a_j\|} \in \mcal{Z}(L^1(G))
 %       \obp \mcal{Z}(A).
%\end{equation*}
\vspace*{2mm}
\noindent (2) Let $G$ be an ${\bf [FIA]}^-$ group. Since  $f \in L^1(G)
\obp \Z(A)$ and  $C_c(G) \ot \Z(A) $ is dense in $L^1(G) \obp \Z(A)$,
there exists a sequence $\{f_n\}$ in $C_c(G) \ot \Z(A)$ such that
$\lim\limits_{n \to \infty} \|f_n - f \|_1 = 0$.  By \Cref{hash}, we
have
\[
\| f_n^{\sharp} - f \|_1 = \| ({f_n -
  f})^\sharp \|_1 \leq \| f_n - f \|_1 \text{ for all } n \in \N;
\]
so that $ \lim\limits_{n \to \infty} \| f_n^{\sharp} - f \|_1 =0
$. From \Cref{centreofL1GA-FC}, we know that the norm of $L^1(G) \obp
\Z(A)$ restricted to $\Z(L^1(G)) \ot Z(A)$ coincides with the norm
coming from $\Z(L^1(G)) \obp \Z(A)$. So, we shall be done if we can
show that $f_n^\sharp \in \Z(L^1(G)) \ot \Z(A)$ for all $n$. It
actually suffices to show that $(h \ot a)^\sharp = h^\sharp \ot a$ for
every $h \in C_c(G)$ and $a \in A$, where the function $\sharp$ on the
right hand side represents the projection from $ L^1(G)$ onto
$\Z(L^1(G))$ as in \Cref{hash} (by taking $A = \C$). This follows
rather easily as
\[
(h \ot a)^\sharp(x) = \int\limits_{\ol{\mathrm{Inn}(G)}} (\beta \cdot
(ha))(x) d\beta = \left(\ \int\limits_{\ol{\mathrm{Inn}(G)}}
h(\beta^{-1}(x)) d\beta\right) a = ( h^\sharp \ot a)(x)
\]
for all $x \in G$.
\end{proof}

Since every compact group is an ${\bf [FIA]}^-$ group (see
\cite[Diagram 1]{palmerpaper}), we deduce the following:

\begin{cor}
  Let $G$ be a compact group and $A$ be a unital Banach algebra. Then,
  \[
  \Z(L^1(G) \obp A) = \Z(L^1(G)) \obp \Z(A).
  \]
  \end{cor}

\bigskip

\subsection*{Acknowledgements}
The authors would like to thank Professor V.~Losert for providing an explanation for one of his results from 
\cite{losert}, which turned out to be instrumental in analyzing
the center of any generalized group algebra.

%\bibitem{PetersSund} J. Peters\ and\ T. Sund, Automorphisms of locally compact groups, Pacific J. Math. {\bf 76} (1978), no.~1, 143--156.

%\bibitem{rieffel} M. A. Rieffel, The Radon-Nikodym theorem for the Bochner integral, Trans. Amer. Math. Soc. {\bf 
% 131} (1968), 466--487.

%M. A. Rieffel, Dentable subsets of Banach spaces, with application to a Radon-Nikod\'{y}m theorem, in {\it Functional Analysis (Proc. Conf., Irvine, Calif., 1966)}, 71--77, Academic Press, London.

%\bibitem{Kum01} A.~Kumar, Operator space projective tensor product of
 % $C^*$-algebras, Math. Z., 237, 2001, 211 - 217.

\end{document}